\documentclass[
a4paper, %
fleqn, %
reqno, %
12pt %
]{amsart}

\usepackage{fullpage}
\usepackage{graphicx}
\usepackage{multicol,multirow}
\usepackage{amsmath,amssymb,amsfonts}
\usepackage{mathrsfs}
\usepackage{mathtools}
\usepackage[utf8]{inputenc}
\usepackage[T1]{fontenc}
\usepackage{layout}
\usepackage{graphicx}
\usepackage{paralist}
\usepackage{booktabs}
\usepackage[english]{babel}
\usepackage{mathptmx}

\newcommand{\pol}{\mathcal{P}} %space of polynomail
\newcommand{\is}{\texttt{\rm\footnotesize CND}} %space of interpolation conditions
\newcommand{\NN}{\mathbb{N}} %natural integers
\newcommand{\RR}{\mathbb{R}} %real numbers
\newcommand{\CC}{\mathbb{C}} %complex numbers
\newcommand{\KK}{\mathbb{K}} %either $\RR$ or $\CC$

\newcommand{\lag}{\mathbf{L}}% Lagrange interpolation operator
\newcommand{\tay}{\mathbf{T}}% Taylor operator
\newcommand{\pro}{\mathbf{P}}% ortho projection
\newcommand{\kg}{\mathbf{K}}% Kergin interpolation operator
\DeclareMathOperator{\ntimes}{\otimes_{\mathrm N}}
\newcommand{\hh}{\mathop{\hbox{$\mathcal{H}$}}} %holomorphic functions
\newcommand{\gs}{\mathop{\hbox{$\mathcal{E}$}}} %the generic space E
\newcommand{\cc}{\mathop{\hbox{$\mathcal{C}$}}} %continuous functions
\newcommand{\vect}[1]{{\text{\rm span}}\left\{#1\right\}} %linear span

\newtheorem{theorem}{Theorem}[section]

\newtheorem{corollary}[theorem]{Corollary}
\newtheorem{lemma}[theorem]{Lemma}

\theoremstyle{remark}
\newtheorem{remark}{Remark}
\newtheorem{problem}{Research problem}

\newtheorem{step}{Step}

\theoremstyle{definition}
\newtheorem{definition}[theorem]{Definition}
\newtheorem{example}[theorem]{Example}
\numberwithin{equation}{section}
\numberwithin{figure}{section}

\begin{document}

\title[The Newton Product of Polynomial Projectors]{The Newton Product of Polynomial Projectors. Part 2 : approximation properties}

\author{François Bertrand}
\address{Institut de Mathématiques de Toulouse,
Université de Toulouse III and CNRS (UMR 5219), 31062, Toulouse Cedex 9, France}
\author{Jean-Paul Calvi}
\address{Institut de Mathématiques de Toulouse,
Université de Toulouse III and CNRS (UMR 5219), 31062, Toulouse Cedex 9, France} 

\date{\today}



\begin{abstract}We prove that the Newton product of efficient polynomial projectors is still efficient. Various polynomial approximation theorems are established involving Newton product projectors on spaces of holomorphic functions on a neighborhood of a regular compact set, on spaces of entire functions of given growth and on spaces of differentiable functions. Efficient explicit new projectors are presented.\end{abstract}

\maketitle

\section{Introduction }
In a recent paper \cite{bercal1} we introduced a new way of forming an approximation operator, precisely a polynomial projector, acting on spaces of functions defined on a (subset of a) space of dimension $n$ out of two polynomials projectors on spaces of functions defined on a space of smaller dimension $n_s$, $s=1,2$, $n=n_1+n_2$. Our process was called the \textbf{Newton product} of the (smaller dimensional) projectors.  It is related but considerably different from the classical method based on the tensor product of operators. The construction of the Newton product relies on a suitable graduation, called a \textbf{Newton structure}, of the interpolation conditions that define the projector. For instance, in the simplest case of a (univariate) Lagrange polynomial projector $\lag[a_0,\dots,a_d]$ with respect to the $d+1$ points $a_0, \dots,a_d$, the specification of a (useful) Newton structure is equivalent to the specification of an ordering of the interpolations points. In fact, such a specification is already required in the Newton formula (based on divided differences) for classical Lagrange interpolation and there lies the origin of our terminology. In general, see below, a Newton structure is obtained in a more general way but, as a first approximation, the reader may retain that as far as the projectors are determined by a set of points (for instance, Lagrange, Kergin, Hakopian projectors) a (natural) Newton structure is determined by an ordering of these points. For example, if we stay with the ordering of $A=\{a_0, \dots,a_d\}$ and $B=\{b_0, \dots,b_d\}$ induced by the indexes, 
\begin{equation}\label{eq:bir}
\lag[a_0,\dots,a_d]\ntimes \lag[b_0,\dots,b_d], 
\end{equation}

(where $\ntimes$ indicates the Newton product) is the bivariate Lagrange interpolation projector at the points $(a_i,b_j)$, $i+j\leq d$; this is a set of interpolation points first considered a long time ago by Bierman \cite{bie}. The main goal of the present work is to validate the previous one in answering in the affirmative the obvious question : if the partial projectors are efficient (in a certain sense) approximation operators, will it be the same for their Newton product ? Actually, we tried to write the present paper in order that it could work in the other direction : we relied as little as possible on the previous work in order that the reader interested by the results obtained in this one might decide to go back to \cite{bercal1} for a deeper understanding of the underlying algebraic machinery. The above question is answered positively for projectors on classical spaces of holomorphic functions on a neighborhood of a regular compact set (in the sense of pluripotential theory), on spaces of entire functions of given growth, on spaces of differentiable functions, the results in the latter case being however less precise. The proofs rest on an identical principle that is explained in the next section. 
\par\smallskip
For the convenience of the reader and make the paper self contained, we will conclude this introduction with a definition of the Newton product as well as a few important examples.
\par
The symbol $\oplus$ denotes a direct sum of vector spaces, $\otimes$ is used for the usual tensor product, while the symbol $\ntimes$ is reserved for the Newton product.
\par   
A polynomial projector $\Pi$ of degree $d$ on a space of (real or complex) functions $\gs=\gs(X)$ defined on $X \subset\KK^n$, $\KK =\RR$ or $\KK=\CC $, is a continuous linear map on $\gs$ with values in the $\gs$-subspace $\pol_d(\KK^n)$ of real or complex polynomials of total degree at most $d$, such that $\Pi\circ\Pi=\Pi$. We set 
\[\is(\Pi)=\big\{\nu\in {\gs}' \;:\; \nu(f)=\nu(\Pi(f)),\; \text{all $f\in \gs$}\big\}\]
where ${\gs}'$ is the space of continuous linear forms on $\gs$. We call $\is(\Pi)$ the \textbf{space of interpolation conditions} for $\Pi$. For instance if $\Pi$ is a Lagrange interpolation projector then $\is(\Pi)$ is the space of functionals spanned by the Dirac functionals $\delta_a : f \to f(a)$ when $a$ runs among the interpolation points.  A list of the classical spaces $\gs$ that that will be considered can be found in Table \ref{tab:functionspaces} below. 
\par
A \textbf{Newton structure} for $\Pi$ is a direct sum of subspaces $J_i\subset \is(\Pi)$, that is,  \[\is(\Pi) = J_0 \oplus J_1 \oplus \cdots \oplus J_d,\] such that, for $j=0, \dots, d$, there exists a projector $\Pi_j$ of degree $j$ with 
\[\is (\Pi_j)=J_0 \oplus J_1 \oplus \cdots \oplus J_j.\]
It is shown in \cite{bercal1} that all projectors possess Newton structures.  In fact, to obtain a Newton structure is equivalent to find a basis $(\mu_\alpha\;:\; |\alpha|\leq d)$ where $|\alpha|$ denotes the length of $\alpha\in \NN^n$ such that for $j=0, \dots, 
d$, the sub-list $(\mu_\alpha\;:\; |\alpha|\leq j)$ is linearly independent on 
$\pol_j(\CC^n)$ and the link with the previous definition is given by 
\[\is(\Pi_j)=  \vect{\mu_\alpha\;:\; |\alpha|\leq j}\quad\text{or}\quad J_j=\vect{\mu_\alpha\;:\; |\alpha|= j}, \quad j=0,\dots, d.\]          
A polynomial projector together with a specific Newton structure is called a \textbf{Newton-structured projector}. It is important not to confuse a polynomial projector with the richer notion of Newton-structured polynomial projector. To distinguish a mere projector from a Newton structured projector, we use the notation \[[\Pi]=(\Pi_0, \dots, \Pi_d)\] to denote the latter. 
\par\smallskip
Now, given two Newton structured projectors $[\Pi^i]$, $i=1,2$, on $\gs(X_i)$ with respective Newton structure corresponding to the basis $(\mu^i_\alpha\;:\; |\alpha|\leq d)$, it is shown in \cite{bercal1} that there exists a unique projector $\Pi$ on $\gs(X_1\times X_2)$ such that 
\[\is(\Pi) =\vect{\mu^1_\alpha \otimes \mu^2_\beta \;:\; |\alpha|+|\beta|=0, \dots, d}.\]
This projector $\Pi$ is called the \textbf{Newton product} of $[\Pi^1]$ and $[\Pi^2]$ and is denoted by $[\Pi^1]\ntimes [\Pi^2]$.
\par
For the sake of notational simplicity, we will sometimes (abusively) write
$\Pi^1\ntimes \Pi^2$ instead of $[\Pi^1]\ntimes [\Pi^2]$, especially when
the Newton structure we use is clear in the context but, in any case, it must
be remembered that the use of the symbol $\ntimes$ (for the Newton product) 
implies that Newton structures have been chosen. The main algebraic formula
regarding the computation of a Newton product on a product function is recalled
in Theorem \ref{th:fbsnppppart2} below.
\par\smallskip 
We will now present several examples of classical or new projectors that are Newton products. In the table below, the Newton structures used appears through the basis $\mu^i_\alpha$ indicated. We omit the mention of the natural spaces $\gs$ involved. Some of these examples are considered in more details further in this work. 

\begin{enumerate}
\item If, for $i=1,2$, $[\Pi^i]$ is the \textbf{Taylor projector} at $a^i\in \KK^{n_i}$ to the order $d$ Newton structured by \[\mu^i_\alpha(f) =D^\alpha f(a^i),\] then $\Pi^1\ntimes\Pi^2$ is the \textbf{Taylor projector} at $(a,b)\in \KK^{n}$, $n=n_1+n_2$, to the order $d$ with
 \[\mu^1_{\alpha} \otimes \mu^2_{\beta}(f) =D^{(\alpha, \beta )} f(a), \quad a=(a^1,a^2).\]
\item\label{exe:projmaesure} Assume that, for $i=1,2$, $[\Pi^i]$ is the $L^2(\mathbf{dm}_i)$-\textbf{orthogonal projection} on $\pol_d(\CC^{n_i})$ Newton structured by 
\[\mu^i_\alpha(f) =  \int f\overline{b}(\alpha,\mathbf{dm}_i, \cdot)\mathbf{dm}_i,\]
where $\mathbf{dm}_i$ is a (sufficiently dense) positive Borel measure supported on a compact set $K_i\subset \CC^{n_i}$ and $b(\alpha, \mathbf{dm}_i , \cdot)$ denotes the usual orthonormal basis (with leading monomial $z^\alpha$) in $L^2(\mathbf{dm}_i)$ obtained by the Gram-Schmidt algorithm from the standard monomials basis ordered with the graded lexicographic order. Then $\Pi^1\ntimes\Pi^2$ is the $L^2(\mathbf{dm}_1\times \mathbf{dm}_2 )$-\textbf{orthogonal projection} on $\pol_d(\CC^{n})$, $n=n_1+n_2$, with
\[\mu^1_\alpha \otimes \mu^2_\beta(f) = \\ \iint f\overline{b}\big((\alpha,\beta), \mathbf{dm}_1\otimes \mathbf{dm}_2, \cdot\big)\mathbf{dm}_1\otimes \mathbf{dm}_2.\]
\item Assume that, for $i=1,2$, $[\Pi^i]$ is the \textbf{Lagrange interpolation projector} $\lag[A^i]$ in $\pol_d(\CC^{n_1})$ at the (unisolvent) set of interpolation points $A^i=\{a^i_\alpha\;:\;|\alpha|\leq d\}$ structured by
\[\mu^i_\alpha(f) =f(a^i_\alpha),\] 
where the interpolation points $a^i_\alpha$ are ordered in such a way that the sets $A^i_j=\{a^i_\alpha\;:\; |\alpha|\leq j\}$ are unisolvent for Lagrange interpolation in $\pol_j(\CC^{n_i})$, $j=0,\dots,d$. Then $[\Pi^1]\ntimes_N[\Pi^2]$ is the \textbf{Lagrange interpolation projector} in $\pol_d(\CC^{n})$, $n=n_1+n_2$, at the (unisolvent) set of interpolation points 
\[A=\left\{\big(a^1_\alpha,a^2_\beta\big)\;:\; |\alpha|+|\beta|\leq d\right\}\subset \CC^{n}.\] 
We refer to \cite{jpcaicm} for details. 
\item Assume that $[\Pi^1]$ is is the \textbf{Lagrange interpolation projector} $\lag[A^1]$ structured as above and $[\Pi^2]$ is the \textbf{Kergin interpolation projector} $\kg[A^2]$ in $\pol_d(\CC)$ with $A^2=\{z_0, \dots,z_d\}\subset \CC^{n_2}$ structured by, see \cite{bercal1} for details,
\[\mu^2_\alpha(f) =\int_{S_{|\alpha|}}D^\alpha f \big(z_0+ \sum_{i=1}^{|\alpha|}t_i(z_i-z_0)\big) dt,\]
where $S_k=\{(t_1,\dots,t_k)\in [0,1]^k \;:\; \sum_{i=1}^k t_i \leq 1\}$ denotes the ordinary simplex in $\RR^{k}$ (and $dt$ the ordinary Lebesgue measure on it) then 
\[[\Pi^1]\ntimes [\Pi^2]= \lag[A^1]\ntimes \kg[A^2]\]
is the projector the space of interpolation conditions is spanned by the functionals
\[f \to \int_{S_{|\beta|}}D^{(0,\beta)} f \big(a^1_\alpha, z_0+ \sum_{i=1}^{|\beta|}t_i(z_i-z_0)\big) dt, \quad |\alpha|+|\beta|\leq d,\]
where $D^{(0,\beta)}$ indicates derivation with respect to the last $n_2$ complex variables. 
As will be shown in the sequel, such a mixed Newton product seems to be particularly useful in the case $n_1=1$ and $n_2=2$. 

\end{enumerate}

%%%%%%%%%%%%%%%%%%%%%%%%%%%%%%%%%%%%%%%%%%%%%%%%%%%%%%%%%%%%%%%%%%%%%%%%%%%%%%%%%%%%%%%%%%%%%%%%%%%%%%%%%%%%%%
\section{Stating the problem}
\subsection{Newton sequences} For each $d\in \NN$, we let $\Pi_d$ denote a polynomial projector of degree $d$ on a space of (real or complex) functions $\gs$. The sequence $\mathcal{N}=(\Pi_0,\Pi_2, \dots)$ will be called a \textbf{Newton sequence} on $\gs$ if : 
\begin{align} \label{eq:newstruwithN0}
&\text{For all $d\in \NN$, $\is(\Pi_d) \subset \is(\Pi_{d+1})$,} 
\\ 
\intertext{or, equivalently,}\label{eq:newstruwithN}
&\text{$(\Pi_0,\dots, \Pi_d)$ defines a Newton structure for $\Pi_d$. }
\end{align}
\par 
We shall denote the Newton-structured projector in \eqref{eq:newstruwithN}
by $[\mathcal{N}]_d$. 
\par 
Given such a Newton sequence $\mathcal{N}$, we set $\pi_0=\Pi_0$ and, for $k\geq 1$, 
$\pi_k=\Pi_{k}-\Pi_{k-1}$. Then $\pi_k$ is the $k$-th \textbf{Newton summand} 
for $[\mathcal{N}]_d$ and this for any $d\geq k$, so that 
\begin{equation}\label{eq:withnewsum} \Pi_d =\sum_{k=0}^d \pi_k, \quad  d\in \NN.
\end{equation}
Likewise, there exists a sequence of spaces $J_i$, $i\in \NN$, such that
\begin{equation}\label{eq:withspacesum}\is (\Pi_d)=\bigoplus_{0\leq i \leq d} J_i, \quad d\in \NN,\end{equation}

\begin{example} We use the projectors presented in the introduction.
 \par\smallskip
\begin{asparaenum}[(A)]
\item The sequence for which $\Pi_d$ is the \textbf{Taylor projector} at a point $a$ to the order $d$, $\Pi_d=\tay_a^d$, is a basic example of Newton sequence for which $J_i=\text{span}\{ f \to D^{|\alpha|}f(a)\;:\; |\alpha|=i\}$. 
 \par\smallskip
\item Another fundamental example is furnished by the \textbf{orthogonal projectors} $\pro_{d, \mathbf{dm}}$ with respect to a sufficiently dense measure $\mathbf{dm}$ where $J_i=\text{span}\{ f \to \langle f \,,\, b (\alpha, \mathbf{dm}, \cdot)\rangle\;:\; |\alpha|=i\}$, see Example \ref{exe:projmaesure} in the introduction. 
 \par\smallskip
\item A Newton sequence of \textbf{Lagrange interpolation projectors} will be obtained with $\Pi_d=\lag[A_d]$ with the condition that $A_d\subset A_{d+1}$ for $d\in \NN$ where $A_d$ is the (unisolvent) set of interpolation points for $\lag[A_d]$, for which $J_i=\text{span}\{ f \to f(a)\;:\; a\in A_i\setminus A_{i-1}\}$, $i\geq 1$. 
 \par\smallskip
\item Likewise a Newton sequence of \textbf{Kergin interpolation projectors} will be obtained with $\Pi_d=\kg[A_d]$ with the same condition that the set of interpolation points are nested, that is $A_d\subset A_{d+1}$ for $d\in \NN$. Typically, starting from a sequence $a_d$ of points, we will take $A_d=\{a_0,\dots,a_d\}$ so that $\Pi_d=\kg[a_0,\dots,a_d]$, $d\in \NN$, see the introduction for a description of the spaces $J_i$ and \cite{bercal1} for further details. 
\end{asparaenum}
\end{example}
\begin{definition}\label{def:convergingstuff} Let $\mathcal{F}$ be a (topological vector) space of functions containing $\gs$ (hence also the polynomials) such that the injection $f\in \gs \to f\in \mathcal{F}$ is continuous. Note that $\Pi_d$ is continuous as a linear map from $\gs$ to $\mathcal{F}$. 
We say that a \textbf{Newton sequence} $\mathcal{N}=(\Pi_0,\Pi_1, \dots)$ on $\gs$ is \textbf{$\mathcal{F}$-converging} when, for all $f\in \gs$, $\Pi_d(f)$ converges to $f$ in $\mathcal{F}$ as $d\to \infty$. When precision is needed we say that $\mathcal{N}$ is \textbf{$\mathcal{F}$-converging} on $\gs$. When $\mathcal{F}=\gs$ we just say that
$\mathcal{N}$ is \textbf{converging} on $\gs$. 
If convergence holds only for all functions in a subspace $E$ of $\gs$, we say that $\mathcal{N}$ is \textbf{$\mathcal{F}$-converging} on $E\subset \gs$.
\end{definition}  
Table \ref{tab:functionspaces} collects the various spaces $\gs$, $E$ and $\mathcal{F}$ that will be used in the sequel. 
\begin{table}[h]
\begin{tabular}{p{0.6\textwidth}p{0.3\textwidth}} \toprule Space $\gs$ or space $E$ & 
Space $\mathcal{F}$\\ \midrule 
$\hh(K)$, the space of holomorphic function on a neighborhood of (a regular polynomially convex compact set) $K$. & $\cc(K)$ the space of continuous functions on  $K$, then $\hh(K)$. \\ \midrule
$\hh(\CC^n)$, the space of entire functions and $E$ a subspace of entire
functions of given growth. & $\hh(\CC^n)$. \\ \midrule
${\cc}^m(K)$ the standard space of $m$ times continuously differentiable functions on the interior of $K$ (a fat compact set, see Section \ref{sec:smooth}, whose all derivatives of order $\leq m$ extend continuously to $K$. & $\cc(K)$.\\ \bottomrule \\
\end{tabular}
\caption{Function spaces, all endowed with their usual topology.}
\label{tab:functionspaces}
\end{table}
\par\smallskip
Classical results in approximation theory are naturally expressed in the above terminology. For instance, the Newton sequence the $d$-th element of which is the Taylor projector at the origin $0$ to the order $d$ is converging on the space $\hh(\{0\})$ formed of holomorphic functions on a neighborhood of the origin (endowed with its usual limit inductive topology).  
Classical univariate Lagrange interpolation theory is concerned with the search for conditions on a given sequence $(a_0, a_1, a_2, \dots)$ ensuring that $(\Pi_0,\Pi_2, \dots)$ with $\Pi_d=\lag[a_0,\dots,a_d]$ is $\mathcal{F}$-converging on $\gs$ where $\gs$ is a classical space of smooth functions on an interval $[a,b]$, and $\mathcal{F}=\cc([a,b])$, or
$\mathcal{F}=\hh(K)$, see \cite{walsh, smirnov, gaier}, or $E\subset \gs=\mathcal{F}=\hh(\CC)$ is a space of entire functions of given growth, see \cite{guelfond}.
\subsection{The Newton product of Newton sequences}
By using the Newton product of polynomial projectors, we can easily construct a natural Newton sequence on $\gs (X_1\times X_2)$ when we are given two Newton sequences $\mathcal{N}^1=(\Pi^1_0,\Pi^1_1, \dots)$ on $\gs(X_1)$ and $\mathcal{N}^2=(\Pi^2_0,\Pi^1_2, \dots)$ on $\gs(X_2)$. Here, by writing $\gs(X_s)$ we emphasize that the functions in $\gs$ are defined on $X_s$. The connection between $\gs(X_s)$ and $\gs(X_1\times X_2)$ will be obvious in all cases considered. 
\par 
For all $d\in \NN$, we compute the Newton product, see \eqref{eq:newstruwithN},
\begin{equation}\label{eq:defpronewseq}%
\Pi_d:= [\mathcal{N}^1]_d \ntimes [\mathcal{N}^2]_d = (\Pi^1_0,\dots, \Pi^1_d)\ntimes (\Pi^2_0,\dots, \Pi^2_d).
\end{equation}
In fact if $(J^s_i)$, $s=1,2$ is the sequence of spaces associated to $\mathcal{N}^s$ as in \eqref{eq:withspacesum}, according to \cite[Corollary 4.8]{bercal1}, we have 
\begin{equation}\label{eq:isforprod} \is(\Pi_d)= \bigoplus_{i+j\leq d} J^1_i\otimes J^2_j \end{equation}
which the reader may read as a characterization property of $\Pi_d$.
\par
The above construction provides a sequence of polynomials projectors and this sequence is actually a Newton sequence itself . 
\begin{lemma} The sequence $(\Pi_0, \Pi_1, \dots)$ where $\Pi_d$ is defined as in \eqref{eq:defpronewseq} is a Newton sequence on $\gs(X_1\times X_2)$.  
\end{lemma}
\begin{proof} According to \eqref{eq:newstruwithN0}, we need to prove that $\is(\Pi_d) \subset \is(\Pi_{d+1})$ for all $d\in \NN$.  
Since $\mathcal{N}^s$ is a Newton sequence, there exists a sequence of spaces $J_i^s$, $i\in \NN$, see \eqref{eq:withspacesum}, such that
\[\is (\Pi_k^s)=\bigoplus_{0\leq i \leq k} J^s_i, \quad d\in \NN, \quad s=1,2.\] In view of \eqref{eq:defpronewseq} and using
\eqref{eq:isforprod}  for both equalities, we have 
 \[\is(\Pi_d)= \bigoplus_{i+j\leq d}  J^1_i\otimes J^2_j \subset \bigoplus_{i+j\leq d+1} 
 J^1_i\otimes J^2_j = \is(\Pi_{d+1}). \qedhere
\] \end{proof}
We are now in position to fix the terminology that will be used in this paper. 
\begin{definition} The above Newton sequence will be denoted by $\mathcal{N}=\mathcal{N}^1\ntimes \mathcal{N}^2$ and called the \textbf{Newton product} of $\mathcal{N}^1$ by $\mathcal{N}^2$. In accordance with the terminology introduced for the Newton product, $\mathcal{N}^1$ (resp. $\mathcal{N}^2$) will be called the left (resp. right) \textbf{divisor} of $\mathcal{N}$.
\end{definition} 
 Now, the obvious approximation problem is as follows. 
\begin{problem}\label{prob:basic} Suppose that $\mathcal{N}^s$ is a $\mathcal{F}_s$-converging Newton sequence on $\gs(X_s)$, $s=1,2$, is $\mathcal{N}^1\ntimes \mathcal{N}^2$ a $\mathcal{F}$-converging sequence on $\gs(X_1\times X_2)$ where $\mathcal{F}$ naturally depends on $\mathcal{F}_s$, $s=1,2$ ? 
\end{problem}
We will show that the answer is generally positive, thus showing that the Newton product leads to the construction of new effective approximation projectors. Various explicit examples will be given in the sequel. 
\subsection{Strategy of proof}\label{sec:strategy}
We describe the simple general strategy that we follow in answering Problem \ref{prob:basic} for different spaces. We felt it preferable to provide distinct proofs in the three main cases that we study rather than to search for a very general statement which would require assumptions whose verifications would require essentially the same amount of work. 
\par 
\smallskip
Assume that we work with $\gs(X_1)$, $\gs(X_2)$ and $\gs(X_1\times X_2)$ with $X_s\subset \RR^{n_s}$ or $\CC^{n_s}$. We use the notation introduced in the previous subsection, in particular $\Pi_d$ is defined as in \eqref{eq:defpronewseq}.
\begin{step} We find product polynomials 
\begin{equation}\label{eq:defpalphabeta} p_{\alpha,\beta}(z^1,z^2)=p^1_{\alpha}(z^1)p^2_{\beta}(z^2)\quad\text{with $\deg p^1_{\alpha} =|\alpha|$ and $\deg p^2_{\beta} =|\beta|$},\end{equation} such that, for all $f\in \gs(X_1\times X_2)$, we have 
\begin{equation}\label{eq:strat1} f=\sum_{j=0}^\infty \sum_{|\alpha|+|\beta|=j} c_{\alpha\beta}(f) p_{\alpha,\beta}\end{equation}
where  the $c_\alpha$ are functionals on $\gs(X_1\times X_2)$ and convergence holds  in $\gs(X_1\times X_2)$ (hence	also in $\mathcal{F}$).
\end{step}
\begin{step} Since $\Pi_d$ is a projector of degree $d$, it coincides with the identity on the polynomials of degree $\leq d$ so that we have 
\[\Pi_d\left(\sum_{j=0}^d \sum_{|\alpha|+|\beta|=j} c_{\alpha\beta}(f) p_{\alpha,\beta}\right) = 
\sum_{j=0}^d \sum_{|\alpha|+|\beta|=j} c_{\alpha\beta}(f) p_{\alpha,\beta}\]
and, consequently, it follows from \eqref{eq:strat1}, together with the continuity of $\Pi_d : \gs(X_1\times X_2)\to \mathcal{F}$, that 
\begin{align}\label{eq:strat2}
f-\Pi_d(f)&= \sum_{j=d+1}^\infty \sum_{|\alpha|+|\beta|=j} c_{\alpha\beta}(f) \big\{p_{\alpha,\beta}-\Pi_d(p_{\alpha,\beta})\big\} \\
\label{eq:strat3}
&=\sum_{j=d+1}^\infty \sum_{|\alpha|+|\beta|=j} c_{\alpha\beta}(f) p_{\alpha,\beta} - \sum_{|\alpha|+|\beta|=d+1}^\infty c_{\alpha\beta}(f) \Pi_d(p_{\alpha,\beta}), 
\end{align}
where convergence holds in $\mathcal{F}$. 
Observe that, since the first series tends to $0$ in $\mathcal{F}$ as $d\to \infty$, we might restrict ourselves to show that $\sum_{j=d+1}^\infty \sum_{|\alpha|+|\beta|=j} c_{\alpha\beta}(f) \Pi_d(p_{\alpha,\beta})$ goes to $0$ as $d$ tends to $\infty$ in \eqref{eq:strat3}. It is however equally simple and somewhat more elegant to work with the right hand side of \eqref{eq:strat2}.
\par
\smallskip
At this point, we need an algebraic formula established in \cite{bercal1} for the computation of a Newton product projector applied to a product function  : 
\begin{theorem}\label{th:fbsnppppart2} Let, for $i=1,2$, $[\Pi^i]=(\Pi^i_0,\dots,\Pi^i_d)$ be a Newton structured polynomial projector of degree $d$ on $\gs(X_i)$, $X_i\in \KK^{n_i}$. If $f_i\in \gs(X_i)$ we denote by $f_1f_2$ the product function defined by $(f_1f_2)(x^1,x^2)= f_1(x^1)f_2(x^2)$ then we have 
\begin{equation}\label{eq:sbf}
\Pi (f_1f_2)=\sum_{(i,j)\in\NN_d(2)}\pi_{i}^1(f_1) \pi_{j}^2(f_2), \quad f_i\in \gs(X^i), 
\end{equation} where the $\pi^s_i$ denote the Newton summands corresponding to $\Pi^s$, see \eqref{eq:withnewsum}.
\end{theorem}
\begin{proof} See \cite[Theorem 4.5 (2)]{bercal1}. 
\end{proof}
We may now state the key computational lemma.
\end{step}
\begin{lemma}\label{th:lemmabd} We use the notation above, see in particular \eqref{eq:defpalphabeta}.
Assume that $|\alpha|+|\beta|\geq d+1$ and
let 
\begin{equation}\label{eq:defBd}
B(d,\alpha, \beta) = \{(i_1, i_2) \in \NN^2 \;:\; d + 1\leq  i_1 + i_2;\;  i_1\leq |\alpha|;\; i_2\leq |\beta| \}.
\end{equation}
We have
\[p_{\alpha,\beta}-\Pi_d(p_{\alpha,\beta})= \sum_{(i_1,i_2)\in B(d,\alpha,\beta)} \pi^1_{i_1}(p^1_\alpha)\pi^2_{i_2}(p^2_\beta).\]
\end{lemma}
\begin{proof}
Let $d_1=|\alpha|$ and $d_2=|\beta|$. Since $\Pi_{d_s}$ is a projector of degree $d_s$ we have
\[p_{\alpha,\beta}=p^1_\alpha p^2_\beta= \Pi_{d_1}(p^1_\alpha)\Pi_{d_2} (p^2_\beta),\]
and, in view of \eqref{eq:withnewsum},
\begin{align} \label{eq:theplus}p_{\alpha,\beta} = \left(\sum_{i_1\leq d_1}\pi_{i_1}(p^1_\alpha)\right)\cdot 
\left(\sum_{i_2\leq d_2}\pi_{i_2}(p^2_\beta)\right)
= \sum_{i_1\leq d_1, i_2\leq d_2} \pi_{i_1}(p^1_\alpha) \pi_{i_2}(p^2_\beta). 
\end{align}
On the other hand, since $p_{\alpha,\beta}$ is a product function, the product formula 
in Theorem \ref{th:fbsnppppart2} yields 
\begin{equation}\label{eq:theminus}
\Pi_d(p_{\alpha,\beta})=\sum_{i_1+i_2\leq d} \pi^1_{i_1}(p^1_\alpha)\pi^2_{i_2}(p^2_\beta). 
\end{equation}
The relation follows immediately on subtracting \eqref{eq:theminus} from \eqref{eq:theplus}. 
\end{proof}
\begin{step} The properties of the factors $\mathcal{N}^s$ enter into play when estimating
 $p_{\alpha,\beta}-\Pi_d(p_{\alpha,\beta})$ with the help of the formula given in the previous lemma. 
 To estimate the terms in the sum, we invoke the assumption that the divisor sequences $\mathcal{N}^s$, $s=1,2$, are converging via the use of a uniform boundedness principle (Banach-Steinhaus theorem), \cite[Chapter 2]{rudin}. It is therefore essential to work on spaces where the principle holds. 
\end{step}
% 
%
%%%%%%%%%%%%%%
\section{Spaces of holomorphic functions on a neighborhood of a regular compact set}\label{sec:holomorphic}

\subsection{The space $\hh(K)$} The space $\hh(K)$ formed of all functions holomorphic on a neighborhood of the compact set $K$ is endowed with the usual inductive limit of the spaces $H(\Omega)$ or, equivalently, $A(\Omega)$ where $\Omega$ is an open (bounded) neighborhood of $K$ and $H(\Omega)$ denotes the space of holomorphic functions on $\Omega$ with the topology of uniform convergence on compact subsets, while $A(\Omega)$ is the \emph{Banach space} of the functions continuous on $\overline{\Omega}$ and holomorphic on $\Omega$ with the usual sup-norm on $\overline{\Omega}$. Thus, for instance, a sequence
$f_n$ converge to $f$ in $\hh(K)$ if and only if there exists an open bounded neighborhood $\Omega$ of $K$ such that $f$ and $f_n$ are in $A(\Omega)$ for
 all $n$ and $f_n$ converges uniformly to $f$ on $\overline{\Omega}$. A linear map from $\hh(K)$ onto another topological vector space $\mathcal{F}$ 
is continuous if and only if all its restrictions to the spaces $A(\Omega)$, $\Omega \supset K$ are continuous. 
\subsection{Some tools from pluripotential theory}
We recall some basic facts from pluripotential complex theory which are required in the sequel.  A good general reference is the book by Klimek \cite{klimek}. The survey by Levenberg \cite{levenberg} also contains the required material (and much more) and is more approximation-theory oriented. The non-specialist reader may freely assume that the compact sets $K$ we work with are convex in the ordinary geometrical sense. This would imply a limited loss of generality only, at least from a practical point of view. 
\par 
Let $K$ be a regular polynomially convex compact set in $\CC^n$. Its continuous pluri-subharmonic Green-Siciak extremal function is denoted by $V_K$. Recall that this function can be defined for instance as 
\[V_K(z) = \max\left(0 ,\sup\Big\{ \frac{1}{\deg p} \ln |p(z)|\; :\; p\in \pol(\CC^n),\; \|p\|_K \leq 1\Big\}\right).\]
For $R>1$, we define the bounded open set $K_R=\{z\in \CC^n \;:\; V_K(z) < \ln R\}$. 
\par A first use of these level sets appears in the Bernstein-Walsh-Siciak inequality \cite{klimek} which states that
\begin{equation}\label{eq:BWSinequality} \|p\|_{\overline{K_R}}\leq R^{\deg p}\|p\|_K, \quad p\in \pol_d(\CC^n), \quad R>1.
\end{equation} 
The level sets $K_R$ are  further related to the rate of polynomial approximation of a holomorphic function on $K$. Namely, if 
\[\text{dist}(f, \pol_d(\CC^n))=\inf\{\|f-p\|_K\;:\; p\in \pol_d(\CC^n)\},\]
then $f$ is holomorphic (or extends to a holomorphic function) on $K_R$  if and only if
\[\limsup_{d\to\infty} \text{dist}\Big(f, \pol_d(\CC^n)\Big) ^{1/d} < 1/R.\]
In particular,  setting
\begin{equation}\label{eq:defrhof} 
\limsup_{d\to\infty} \text{dist}\Big(f, \pol_d(\CC^n)\Big) ^{1/d} = \dfrac{1}{\rho(f)},
\end{equation}
a continuous function $f$ on $K$ extends to a function in $\hh(K)$ if and only if $\rho(f) >1$
and, more precisely, for all $R<\rho(f)$, $f$ extends to a function in $A(K_R)$. 
\subsection{Bernstein-Markov measures}\label{sec:BMmeasures}
An asymptotically optimal approximation polynomial for functions in $\hh(K)$, $K$ as above, can be obtained as a Fourier expansion with respect to a suitable probability measure as follows. One says that a probability measure $\mu$ with compact support on $K$ is a \textbf{Bernstein-Markov measure} if, for all $\epsilon >0$, there exists a constant $C(\epsilon)=C(\epsilon, \mu)$ such that for all polynomial $p$ we have
\begin{equation}\label{eq:inBM}\|p\|_K \leq C(\epsilon)(1+\varepsilon)^{\deg (p)} \|p\|_2, 
\quad \|p\|_2=\sqrt{\int_K |p(z)|^2 d\mu(z)}.\end{equation}
In other words, the $L^2$-norm of polynomial behaves asymptotically essentially like the sup-norm on $K$. By a theorem of Nguyen and Zériahi \cite{ntv}, on a regular compact set $K$ as above, such a measure always exists. See also \cite{blolev} for a general discussion on Bernstein-Markov measures. \par  Given such a measure, we denote by $(b_\alpha)$ the orthonormal sequence of polynomials obtained by the Gram-Schmidt process from the monomial sequence $(z^\alpha)$ ordered with respect to the graded lexicographic order. In particular, $z^\alpha$ is the leading monomial in $b_\alpha$ and $\deg b_\alpha =|\alpha|$. When it is necessary to clarify the measure we use, instead of $b_\alpha(z)$, we use the notation $b(\alpha, \mu, z)$ as in the introduction. The corresponding orthogonal projection is
\begin{multline}\label{eq:orthprojBM} 
\pro_{d,\mathbf{dm}}(f)= \sum_{|\alpha|\leq d} c_\alpha(f) b_\alpha, \\\quad c_\alpha(f) = \langle f, b_\alpha\rangle =\int_K f(z)\overline{b}_\alpha(z)\mathbf{dm}(z),\quad 
 f\in \hh(K).
\end{multline}
It is well known, see \cite{zer1}, that, $\mathbf{dm}$ being a Bernstein-Markov measure,
\begin{equation}\label{eq:orthonerlybest} \limsup_{d\to\infty} \{\|f-\pro_{d,\mathbf{dm}}(f)\|_K\}^{1/d}= \limsup_{d\to\infty} \text{dist}(f, \pol_d(\CC^n)) ^{1/d}.\end{equation}
This is easily shown as follows. Taking $t_d$, a best (uniform) approximation of degree $d$ of $f$ on $K$,  i.e. such that $\text{dist}(f, \pol_d(\CC^n)) =\|f-t_d\|_K$, we have 
\begin{align}\|f-t_d\|_K & \leq \|f-\pro_{d,\mathbf{dm}}(f)\|_K \leq C(1+\epsilon)^d \|f-\pro_{d,\mathbf{dm}}(f)\|_{2}\\
&  \leq C(1+\epsilon)^d\|f-t_d\|_2 \leq C(1+\epsilon)^d\|f-t_d\|_K,\end{align}
where we use that the orthogonal projection furnishes the best $L^2$ approximant of $f$ on the second line. Now, \eqref{eq:orthonerlybest} immediately follows.
\par 
Likewise, since, by orthogonality,  
\[ c_{\alpha}(f)= \langle f, b_\alpha\rangle = \langle f-\pro_{|\alpha|-1,\mathbf{dm}}(f), b_\alpha\rangle,\]
by the Cauchy-Schwarz inequality, we have 
\begin{multline}\label{eq:boundcalpha}
|c_{\alpha}(f)|\leq \|f-\pro_{|\alpha|-1,\mathbf{dm}}(f)\|_2 \leq \|f-t_{|\alpha|-1}\|_2 \\ \leq \|f-t_{|\alpha|-1}\|_2\leq \text{dist}(f, \pol_{|\alpha|-1}(\CC^n)).
\end{multline}
Note that, in fact, we have the series expansion
\begin{equation}
f=\sum_{j=0}^\infty \sum_{|\alpha|=j} c_\alpha(f) b_\alpha, \quad f\in \hh(K), 
\end{equation}
where the convergence holds in $\hh(K)$. For this and other applications to pluripotential theory, we refer to \cite{zer1}. 
\par
Observe finally that, when applied with $p=b_\alpha$, inequality \eqref{eq:inBM} reads as
\begin{equation}\label{eq:ineqorthopoly}
\|b_\alpha\|_K \leq C(\epsilon)(1+\varepsilon)^{|\alpha|}, \quad |\alpha|=0,1,2, \dots.
\end{equation}
%
%
%%%%%%%%%%%%%%%%%%%%%%%%%%%%%%%%%%%%%%%%%%%%%%%%%%%%%%%%%%%%%%%%%%%%%%%%%%%%%%%%%%
%
\subsection{Cartesian products}
We now collect a few facts that we will need later about the above notions in relation to Cartesian products of sets. 
\begin{asparaenum}[(i)]
	\item  The Cartesian product $K_1\times K_2 \subset \CC^{n_1}\times \CC^{n_2}$ of two regular polynomially convex sets is still regular polynomially convex and, see \cite{klimek},
	\[V_{K_1\times K_2}(z^1,z^2)=\max \big(V_{K_1}(z^1), V_{K_2}(z^2)\big).\] 
	In particular, 
	\[{(K_1\times K_2)}_R= {K_1}_R \times  {K_2}_R.\]
	\item The product  $\mu_1\times \mu_2$ of a Bernstein-Markov probability measure $\mu_1$ on $K_1$ by a Bernstein-Markov probability measure $\mu_2$ on $K_2$ is a Bernstein-Markov measure on $K_1\times K_2$, see \cite[Lemma 2, p. 290]{blocalTD}. Moreover, as is readily checked, the corresponding orthonormal polynomials satisfy the relation
\begin{equation}\label{eq:prodorthopol} b \left((\alpha^1,\alpha^2), \mu_1\times \mu_2, z \right) = b(\alpha^1, \mu_1, z^1) \times b (\alpha^2, \mu_2, z^2), \quad z=(z^1,z^2).
\end{equation}
\end{asparaenum}
\subsection{The convergence theorem}
Roughly, we prove that the Newton product of projectors that approximate efficiently holomorphic functions on regular compact sets itself efficiently approximates holomorphic function on the Cartesian product of the compact sets. 
\begin{theorem}\label{th:converghK} Let $K_s$ be a regular polynomially convex set in $\CC^{n_s}$ and  $\mathcal{N}^s=(\Pi_0^s, \Pi_1^s, \dots)$ a Newton sequence on $\hh(K_s)$,  $s=1,2$.  If $\mathcal{N}^s$ is converging on $\hh(K_s)$ for $s=1,2$ then $\mathcal{N}^1\ntimes \mathcal{N}^2$ is converging on $\hh(K_1\times K_2)$. 
\end{theorem}
\begin{lemma} With the assumptions of the theorem, for $s=1,2$ and for all $R>1$, there exists a constant $\gamma_s(R)$ such that 
\[\|\Pi^s_d(f)\|_K \leq \gamma_s (R) \|f\|_{\overline{{K_s}_{R}}}, \quad f\in A({K_s}_{R}).\]
\end{lemma}
\begin{proof} Each projector $\Pi^s_d$ is a continuous projector from $\hh(K_s)$ onto $\pol_d(\CC^n)$. Using the definition of the topology on $\hh(K_s)$ and the fact that every norm is equivalent on $\pol_d(\CC^{n_s})$, $\Pi_d$ defines a continuous projector from the Banach space $A(K_R)$ to the Banach space $\cc(K)$ of continuous function on $K$. The fact that $\mathcal{N}^s$ is converging yields that $\Pi^s_d(f)$ converges to $f$ uniformly on $K$ for all $f\in A(K_R)$ hence, 
for such $f$, the sequence $(\Pi^s_d(f))$ is bounded in $\cc(K)$. Now, the claim results from an application of the uniform boundedness principle (Banach-Steinhaus Theorem) on the Banach space $A(K_R)$. 
\end{proof} 

\begin{corollary}\label{eq:corunfiboundedness} Likewise, using the Newton summands, $\pi_d^s$, see \eqref{eq:withnewsum}, for  $s=1,2$ and for all $R>1$, there exists a constant $\delta_s(R)$ such that 
\[\|\pi^s_d(f)\|_K \leq \delta_s (R) \|f\|_{\overline{{K_s}_{R}}}, \quad f\in A({K_s}_{R}).\]
\end{corollary}
\begin{proof} Use the lemma and the fact that $\pi_d^s=\Pi^s_{d+1}-\Pi^s_d$.  
\end{proof} 
\begin{proof}[Proof of Theorem \ref{th:converghK}] We take a Bernstein-Markov measure $\mu_s$ on $K_s$ and consider its product $\mu=\mu_1\times \mu_2$ which is Bernstein-Markov on $K_1\times K_2$, see Subsection \ref{sec:BMmeasures}. Following the strategy explained in Subsection \ref{sec:strategy}, we start from from the expansion 
\begin{equation}\label{eq:profhK_1}%
 f=\sum_{j=0}^\infty \sum_{|\alpha|+|\beta|=j} c_{\alpha\beta}(f) b_{\alpha,\beta}, \quad c_{\alpha\beta}(f)=\iint\limits_{K_1\times K_2} f(z) \overline{b}_{\alpha,\beta}(z)d\mu_1(z^1)d\mu_2(z^2),
\end{equation} 
where $f$ is any fixed element in $\hh(K)$, $z=(z^1,z^2)$, $z^s\in K_s$, and, see \eqref{eq:prodorthopol}, 
\begin{equation} 
b_{\alpha,\beta}(z^1,z^2)=b (\alpha, \mu_1,z^1 )\times b (\beta, \mu_2, z^2).
\end{equation}
Recall that convergence in \eqref{eq:profhK_1} holds in $\hh(K)$. Next, a use of equation \eqref{eq:strat2} and Lemma \ref{th:lemmabd} gives 
\begin{equation}\label{eq:profhK_2}
f-\Pi_d(f)= \sum_{j=d+1}^\infty \sum_{|\alpha|+|\beta|=j} c_{\alpha\beta}(f) \sum_{(i_1,i_2)\in B(d,\alpha,\beta)} \pi^1_{i_1}\big(b (\alpha, \mu_1,\cdot)\big)\pi^2_{i_2} \big(b (\beta, \mu_2, \cdot)\big), 
\end{equation}
where $B(d,\alpha, \beta)$ is defined in \eqref{eq:defBd}. We will use this expression to show that $f-\Pi_d(f)$ converges uniformly to $0$ on $K$. Define as above, 
\[\limsup_{d\to\infty} \text{dist}\big(f, \pol_d(\CC^n)\big) ^{1/d} = \dfrac{1}{\rho(f)},\]
 so that since $f\in \hh(K)$, $\rho(f)>1$. Choose $1< R_1< R_2 <\rho(f)$ and $\epsilon >0$ to be fixed later.  A use of Corollary \ref{eq:corunfiboundedness} with $R=R_1$ gives 
\begin{multline}
\|\pi^1_i\big(b (\alpha, \mu^1,\cdot)\big)\pi^2_j \big(b (\beta, \mu_2, \cdot)\big)\|_K 
\leq \|\pi^1_i\big(b (\alpha, \mu^1,\cdot)\big)\|_{K_1}\|\pi^2_j \big(b (\beta, \mu_2, \cdot)\big)\|_{K_2}\\
 \leq 
\delta_1(R_1) \|b (\alpha, \mu^1,\cdot)\|_{\overline{{K_1}_{R_1}}}\cdot \delta_2(R_1) \|b (\alpha, \mu^1,\cdot)\|_{\overline{{K_2}_{R_1}}},
\end{multline}
Now, applying inequality \eqref{eq:BWSinequality} to the right hand side together with the bound \eqref{eq:ineqorthopoly} for the orthonormal polynomials (with the current $\epsilon$), we get the following estimate :
\begin{multline}
\|\pi^1_i\big(b (\alpha, \mu^1,\cdot)\big)\pi^2_j \big(b (\beta, \mu_2, \cdot)\big)\|_K \\ \leq 
\delta_1(R_1)\delta_2(R_1) R_1^{|\alpha|+|\beta|} C(\epsilon, \mu_1)C(\epsilon, \mu_2)(1+\epsilon)^{|\alpha|+|\beta|}
\\ \leq 
\delta_1(R_1)\delta_2(R_1) R_1^{j} C(\epsilon, \mu_1)C(\epsilon, \mu_2)(1+\epsilon)^{j},
\end{multline} 
where we used $|\alpha|+|\beta|=j$. 
Using, this estimate in \eqref{eq:profhK_2} and setting \[\tau = \delta_1(R_1)\delta_2(R_1)C(\epsilon, \mu_1)C(\epsilon, \mu_2),\] we obtain 
\begin{equation}\label{eq:profhK_3} \|f-\Pi_d(f)\|_K \leq \tau
\sum_{j=d+1}^\infty (R_1(1+\epsilon))^j \sum_{|\alpha|+|\beta|=j}  c_{\alpha\beta}(f)\text{card}(B(d,\alpha, \beta)), 
\end{equation}
where $\text{card}$ is used to denote the cardinality. 
Now, since $R_2< \rho(f)$, in view of 
\eqref{eq:boundcalpha}, there exists $\xi=\xi(R_2)$
\begin{equation}
|c_{\alpha\beta}(f)|\leq \text{dist}(f, \pol_{|\alpha|+|\beta|-1}(\CC^n)) \leq \xi \left(\dfrac{1}{R_2}\right)^{|\alpha|+|\beta|}=\xi \dfrac{1}{R_2^{j}}. 
\end{equation}
(Note that, here, $\xi$ is needed only for notational convenience, to have a bound valid for all $\alpha$ and $\beta$, rather than for $|\alpha|+|\beta|$ large enough.) 
Hence, together with \eqref{eq:profhK_3}, we obtain 
\begin{equation}\label{eq:profhK_4}
\|f-\Pi_d(f)\|_K \leq \tau\xi
\sum_{j=d+1}^\infty \left(\dfrac{R_1(1+\epsilon)}{R_2}\right)^j \sum_{|\alpha|+|\beta|=j}\text{card}(B(d,\alpha, \beta)). 
\end{equation}
It remains to observe that the right hand side sum over $\alpha$ and $\beta$ grows polynomially in $j$. This is readily seen. Indeed, 
\begin{equation}\label{eq:escardBd}%
\text{card}(B(d,\alpha, \beta))\leq (|\alpha|+1) \times (|\beta|+1) \leq (j+1)^2 \quad \text{(since $|\alpha|+|\beta|=j$)}.
\end{equation}
Hence,
\begin{equation}\label{eq:profhK_5} 
\sum_{|\alpha|+|\beta|=j}\text{card}(B(d,\alpha, \beta)) \leq (j+1)^2\binom{n+j}{j-1} =O (j^{n+2}), \quad n=n_1+n_2.
\end{equation}
Now take $\epsilon$ so that $R_1(1+\epsilon) <R_2$. This is possible since $R_1< R_2$ and $\epsilon$ can be taken arbitrarily small. 
In view of \eqref{eq:profhK_5}, the right hand side of \eqref{eq:profhK_4} is the remainder of a converging series and this shows that $\|f-\Pi_d(f)\|_K \to 0$ as $d$ tends to $\infty$. Thus, at this point, using the terminology introduced in definition \ref{def:convergingstuff}, we proved that $\mathcal{N}^1\ntimes \mathcal{N}^2$ is $\cc(K)$-converging on $\hh(K_1\times K_2)$.
%
%.  
\par\smallskip 
Actually, since $\Pi_d(f)$ is a polynomial projector the uniform convergence on $K$ implies the convergence on a compact neighborhood of $K$, hence in $\hh(K)$.
 Such a reasoning is detailed in \cite[Section 4, p. 17]{blocal2}. 
In fact, since $R_1(1+\epsilon)$ can be taken arbitrarily close to $1$ and $R_2$ arbitrarily close to $\rho(f)$, our proof actually shows that
\begin{equation}\label{eq:profhK_6}
\limsup_{d\to\infty}  \|f-\Pi_d(f)\|_K^{1/d}=\dfrac{1}{\rho(f)},
\end{equation}
so that $\Pi_d(f)$ provides an asymptotically optimal approximation. This fact also classically implies convergence in $\hh(K)$. Here is a sketch of the proof. 
The series \[L=\Pi_0(f)+\sum_{d=0}^\infty \Pi_{d+1}(f)-\Pi_d(f)\] is  normally converging on $\overline{K_R}$, for $R < \rho(f)$. Indeed,  by  Bernstein Walsh Siciak inequality, 
\eqref{eq:BWSinequality}
\begin{align*}
\|\Pi_{d+1}(f)-\Pi_d(f)\|_{K_R} & \leq R^{d+1} \|\Pi_{d+1}(f)-\Pi_d(f)\|_{K} \\
& \leq  R^{d+1} \|f-\Pi_d(f)\|_K + R^{d+1} \|f-\Pi_{d+1}(f)\|_K ,
\end{align*}
which in view \eqref{eq:profhK_6} is the general term of a converging series. Now the limit $L$ must equal $f$ on $K_R$ since it coincides with $f$ on $K$. 
\end{proof}
\subsection{Examples}\label{sec:exampleHK} Of course, most classical constructive polynomial approximation results in the complex domain can be used together with the above theorem to get efficient multivariate projectors. We just point out three very natural examples. The first one shows how a well-known result is re-captured with our theorem, the second and third ones provide new projectors which, it seems, deserve particular attention. All the Newton products we consider below are constructed using the Newton structure induced by the ordering of the interpolation points. 
\par\smallskip
\begin{asparaenum}[(A)]
\item For $i=1,\dots,n$, we let $K_i$ denote a regular polynomially convex plane compact set, and $a_d^i$, $d\in \NN$, be a sequence of points on the boundary of $K_i$ such that the discrete measure
\[\mu_d^i=\dfrac{1}{d+1}\sum_{j=0}^d [a_d^j]\]
converges weakly to the \emph{equilibrium measure} on $K_i$. 
Then, for every $f\in \hh(K_1\times \cdots \times K_n)$, we have
\begin{equation}\label{eq:prolag}%
\lag[a_0^1,\dots,a_d^1]\ntimes \lag[a_0^2,\dots,a_d^2]\ntimes \cdots \ntimes \lag[a_0^n,\dots,a_d^n] (f)%
\end{equation}
converges to $f$ in $\hh(K_1\times \cdots \times K_n)$. This follows from classical univariate Lagrange interpolation theory, see the references given above, together with (iterated applications of) Theorem \ref{th:converghK}. This well-known result was first published (in another presentation) in \cite{siciak}. Siciak's proof used a multivariate version of the classical (complex) Hermite error formula for Lagrange-Hermite interpolation. Note that the projector on the left hand side of \eqref{eq:prolag} is itself a multivariate Lagrange interpolation $\lag[A^d]$ with the Bierman set of interpolation points, see also the introduction,
\[
A_d=\left\{\left(a_{i_1}^1,a_{i_2}^2, \dots, a_{i_n}^n\right)\;:\; i_1+\dots i_n\leq d\right\}.
\] 
We refer to \cite[Section 6.3, p. 37]{bercal1} for details and earlier references.
\par\smallskip
\item Let $K$ be a convex circular compact set in $\CC^n$ (i.e. $z\in K, \; \theta\in \RR \implies e^{i\theta}z\in K$). If $(a_d)$ is a sequence of points in $K$ such that  
\[\mu_d=\dfrac{1}{d+1}\sum_{j=0}^d [a_j]\] 
converges weakly to a  $e^{i\theta}$-invariant measure $\mu$, that is, such that 
\[\int_K f(z)d\mu(z)= \int_{k} f(e^{i \theta}z)d\mu(z)\quad\text{for all $f\in \cc(K)$ and $\theta\in \RR$},\] 
then the sequence of Kergin projectors $(\kg[a_0,\dots,a_d])$ is converging on $\hh(K)$, see \cite[Theorem 4.1]{blocal1}. \par  Now, assume that, for $s=1,2$, $K_s$ is the product of two convex circular sets, $K=K_1\times K_2$ and that $(a_d^s)$ is a sequence of points in $K_s$ such that  
\[\mu_d^s=\dfrac{1}{d+1}\sum_{j=0}^d [a_j^s]\] 
converges weakly to a  $e^{i\theta}$-invariant measure $\mu_s$ on $K_s$. It is readily seen that the above result applies to the sequence $(a_d)=(a_d^1,a_d^2)$ so that : 
\par\smallskip
\begin{asparaenum}
\item \emph{The sequence of Kergin projectors $(\kg[(a_0^1,a_0^2),\dots,(a_d^1,a_d^2)])$ is converging on $\hh(K)$.} 
\par\smallskip
\item On the other hand, an application of Theorem \ref{th:converghK} gives : 
\begin{theorem} The sequence of projectors \((\kg[a_0^1,\dots,a_d^1]) \ntimes (\kg[a_0^2,\dots,a_d^2])\) is converging on $\hh(K)$. 
\end{theorem}
\end{asparaenum}
\par\smallskip
It is interesting to compare these two converging sequences of Kergin-related projectors. The second has the advantage of interpolating at $\binom{d+1}{2}$ points, see \cite[Section 6.6, p. 42]{bercal1}, whereas the first one interpolates at only $d+1$ points. However, there is a price to pay for this added value, the projectors obtained  with the Newton product no longer preserve (all) homogeneous partial differential relations. For this notion we refer to \cite[Section 6.7]{bercal1} and the references therein. 
\par\smallskip
\item\label{ex:bierm} The next example goes along the same lines but is still more interesting since, in that case, as far as we know, no explicit good interpolation projectors were known. We consider a cylinder, say $D(0,1)\times [-1,1]$ in $\RR^3$. It was shown in \cite{blocal2} that if $(a_d)$ is a sequence of points in the disc $D(0,1)$ such that $\mu_d=\frac{1}{d+1}\sum_{j=0}^d [a_j]$ weakly converges to the (normalized) $d\theta$ measure on the unit circle then the sequence of Kergin projectors $(\kg[a_0,\dots,a_d])$ is converging on $\hh(D(0,1))$ where $D(0,1)$ is regarded as a subset of $\CC^2$. Likewise, if $(b_d)$ is a sequence of points in the interval $[-1,1]$ whose corresponding $\mu_d$ measure converges to the equilibrium measure $dx/\sqrt{1-x^2}$ then classical Lagrange interpolation theory tells that the sequence of Lagrange projectors $(\lag[b_0,\dots,b_d])$  is converging on $\hh([-1,1])$ (where $[-1,1]$ is regarded as a subset of $\CC$). Hence, an application of Theorem \ref{th:converghK} gives :
\begin{theorem}\label{th:exkerlejalagrleja} The sequence of projectors \((\kg[a_0,\dots,a_d]) \ntimes (\lag[b_0,\dots,b_d])\)  is converging on $\hh(D(0,1)\times [-1,1])$. \end{theorem}
The resulting projector interpolates at all points $(a_i,b_j)$ for $i+j\leq d$. We represent these points on the figure below in the case, which is still more interesting as we will see in Section \ref{sec:smooth}, where $(a_d)$ is a Leja sequence for the unit circle and $(b_d)$ is the corresponding $\Re$-Leja sequence, that is, $(b_d)$ is the sequence of the first coordinates of $(a_d)$ in which repeated points are eliminated. \par 
The elements of a Leja sequence for the unit disk can be recursively constructed as follows see \cite[Theorem 5 and Corollary 2]{biacal}
\begin{align} S_1&=(1,-1)\\
S_{2^{n+1}}& = S_{2^n} \wedge \exp \left(\frac{\pi}{2^(n-1)}\right) S_{2^n}, \quad n\geq 1,  \end{align}
where $\wedge$ denotes the usual concatenation operation on tuples. A precise description of $\Re$-Leja sequence is available in \cite{calmanh2}. 
\begin{table}[h!]
\begin{tabular}{p{0.45\textwidth}p{0.45\textwidth}}
		\includegraphics[width=0.4\textwidth]{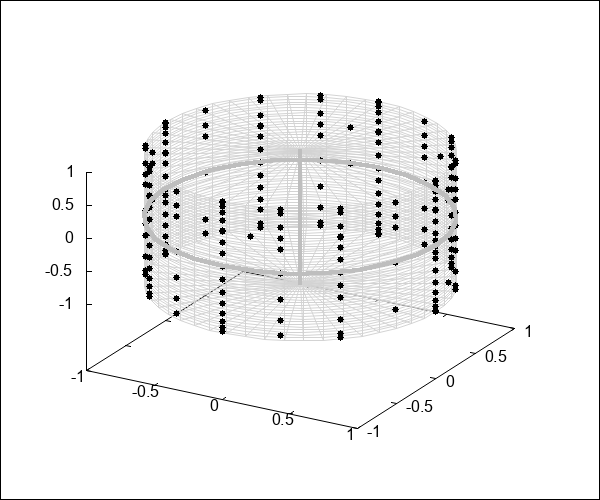} & 
\includegraphics[width=0.4\textwidth]{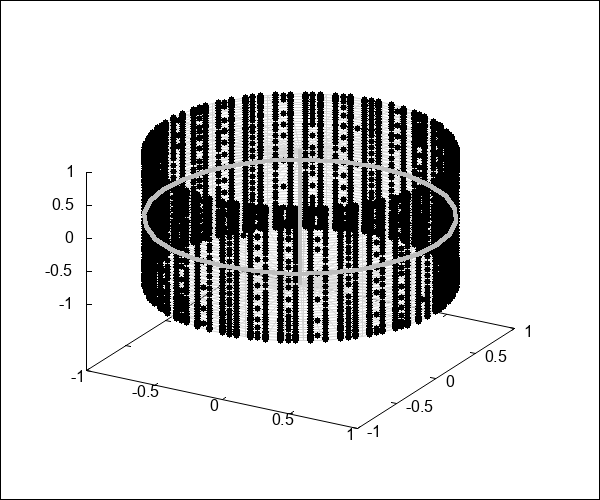} \\
		\begin{center}
		$d=20$
		\end{center} & \begin{center}
		$d=100$
		\end{center}
\end{tabular}
\caption{Interpolation points for a Newton product of Kergin interpolants at Leja points and Lagrange interpolants at $\Re$-Leja points in the cylinder $D(0,1)\times [-1,1]$.}
\label{tab:pointsLeja}
\end{table}
\end{asparaenum} 

\section{Spaces of entire functions}
The effectiveness of a polynomial projector defined on the space of entire 
functions usually depends on the growth (order $\omega$, type $\tau$) of the approximated functions; roughly, the farther the function is from the space of polynomials the stronger are the requirements on the projector. When 
the projectors are univariate Lagrangian projectors at a sequence of points $(a_d)$ (of increasing modulus), the acceptable values for $\omega$ and $\tau$ depend on the velocity at which $|a_d|$ goes to infinity. For instance, a classical result of Polya states that all entire functions of exponential type $<\ln 2$ can be approximated (in $\hh(\CC)$) by Lagrange interpolation at the points $a_d=d$ and the upper bound $\ln 2$ is optimal. Precise statements will be recalled below. As indicated above, the book by Gelfond \cite{guelfond} contains many results on this subject. \par Here, assuming that $\mathcal{N}^s$ is $\hh(\CC^{n_s})$-converging on a space $E_s$ of entire functions on $\CC^{n_s}$ for $s=1,2$, we will look for a space of entire functions $E$ on which $\mathcal{N}^1\ntimes \mathcal{N}^2$ is $\hh(\CC^{n})$-converging, $n=n_1+n_2$.
\par
We will begin by recalling the required material on the study of growth of entire functions of several variables.

\subsection{Growth of entire functions} 
Let $f\in \hh(\CC^n)$ and $N$ a norm on $\CC^n$. We set 
\[
M_N(f,r)=\max_{N(z)\leq r} |f(z)|, \quad r\geq 0.
\]
Given $\omega>0$, the \emph{$\omega$-type} $\tau=\tau(f,\omega,N)$ of $f\in \hh(\CC^m)$ is defined as 
\[
\tau= \limsup_{r\to\infty} \dfrac{\ln M_N(f,r)}{t^\omega}.
\]
When $\tau$ is finite, it is the infimum of all $s$ such that $\ln M_N(f,r) \leq sr^\omega+ O(1)$ as $r\to \infty$. The type may be infinite (when no such $\tau$ exists) and it depends on the norm $N$ we work with. For instance, for $\lambda>0$,
\begin{equation}\label{eq:typelambdaN}
\tau(f,\omega, \lambda N)= \dfrac{1}{\lambda} \tau(f,\omega,N).
\end{equation}
The interesting $\omega$-type is computed when $\omega$ is the \emph{order} of $f$, that is, the infimum of all $s$, if there exists, such that $\ln M_N(f,r)=O(w^r)$ as $r\to \infty$, 
\[
\omega= \limsup_{r\to\infty} \dfrac{\ln\ln M_N(f,r)}{\ln r}.
\]
In contrast to the $\omega$-type, the order of an entire function does not depend on the norm $N$.  A function of finite $1$-type $\tau$ is said to be of \emph{exponential type} $\tau$. 
\par
In the sequel, \begin{quote}\emph{we will assume we work with norms the ball of which are poly-circular, that is to say, $N(z_1,\dots,z_n)\leq 1 \implies N(\lambda_1 z_1,\dots,\lambda_n z_n)\leq 1$ for $|\lambda_i|=1$, $i=1, \dots, n$.}\end{quote}
 For sake of reference, we will call such a norm a \emph{PC-norm}. The common $l_p$ norms satisfy this condition. The usefulness of this condition appears in the lemma below which will enable us to use the strategy described in Subsection \ref{sec:strategy} by using the standard power series expansion as the starting expansion \eqref{eq:strat1}. Namely the lemma shows how the growth of an entire function is reflected into the behavior of the coefficients of its power series expansion. At the end of this section, we briefly explain a way to circumvent the condition for the norm to be PC, at least in certain important cases. 

\begin{lemma}\label{th:estimacoefpoweseries}
Let $f(z)=\sum_{\alpha\in \NN^n} a_\alpha z^\alpha$ be an entire function and $N$ a PC-norm on $\CC^n$. We have 
\begin{equation}\label{eq:estimacoefpoweseries}
|a_\alpha| \leq t^{-|\alpha|} \dfrac{M_N(f,t)}{\delta_N(\alpha)}, \quad t > 0, 
\end{equation} 
where 
\begin{equation}\label{eq:deltaN}%
\delta_N(\alpha)=\max\{|z^\alpha|\,:\, N(z)\leq 1\}, \quad \alpha\in\NN^n.
\end{equation}%
\end{lemma}

\begin{proof}
The reasoning is taken from \cite{ronkin}. Let $P_k(z)=\sum_{|\alpha|=k} a_\alpha z^\alpha$ and assume that, for a fixed $k$, $C_k=\max\{|P_k(z)\,:\, N(z)=1\}$ is attained at $z=u$. Applying the Cauchy inequalities to the univariate function $g(w)=f(wu)=\sum_{k=0}^\infty P_k(u)w^k$, we get
\[
C_k \leq t^{-k}\max\{|g(w)|\;:\;|w|=t\} \leq t^{-k}M_N(f,t), \quad t>0.
\]
Now, the multivariate Cauchy inequalities applied to the polynomial $P_k$ give
\begin{multline*}
|a_\alpha|\leq r^{-\alpha} \max\{\left|P_k(u_1,\dots,u_n)\right|\;:\; |u_i|=r_i,\; i=1,\dots, n\}, \\ \quad r=(r_1,\dots,r_n), \quad r_i>0. 
\end{multline*}
We apply this inequality with 
\[
r\in L=\{(|u_1|, \dots, |u_n|)\;:\; N(u_1,\dots,u_n)\leq 1\}\subset\{N(u)\leq 1\},
\]
where the inclusion holds because $N$ is a PC-norm, thus arriving to
\[
|a_\alpha| r^\alpha \leq C_k, \quad |\alpha|=k, \quad r\in L.
\]
Inequality \eqref{eq:estimacoefpoweseries} is now obtained by passing to the supremum over $r$ in $L$ on the left hand term. 
\end{proof}
It is convenient to introduce the following spaces of entire functions. 
\begin{definition}
Let $N$ be a PC-norm on $\hh(\CC^m)$,  $\omega\geq 0$ and $A>0$. We denote by $E_\omega^m(A,N)$ the subspace of entire functions on $\CC^m$ for which there exists a constant $M$ such that
\[% 
M_P(f,r) \leq M  \exp(Ar^\omega)), \quad r\in \RR.
\]
\end{definition}
There is an obvious connection between these spaces and the classical notions of order and type recalled above. We state it as a remark for future reference.
\begin{remark}\label{typeandspace}
Given $\omega >0$ and $\tau <\infty$, the following statements are equivalent.
\begin{enumerate}
\item $f\in \hh(\CC^m)$ is $\omega$-finite type $< \tau$ with respect to the norm $N$.
\item $f\in \cup_{A <\tau} E_\omega^m(A,N)$.
\end{enumerate}
\end{remark}
We define a norm on $E_\omega^m(A,N)$ 
by setting
\[%
\|f\|_{A,N}^\omega=\sup_{r>0} M_N(f,r)\exp(-Ar^\omega). 
\]
Observe that for all compact $K$ in $\CC^m$, there exists a constant $C(K)$ such 
that all $f\in E_\omega^m(A,N)$, 
\begin{equation}\label{eq:finer}
\|f\|_K\leq C(K) \|f\|_{A,N}^\omega,
\end{equation}
so that convergence with respect to the norm $\|\cdot\|_{A,N}^w$ implies uniform convergence on all compact subsets of $\CC^m$. From this, we deduce the following lemma whose standard proof is only sketched. 
\begin{lemma} 
When endowed with the norm $\|\cdot\|_{A,N}^w$, the space $E_w^m(A,N)$ is a Banach space.
\end{lemma}
\begin{proof}
Let $(f_n)$ be a Cauchy sequence in $E_w^m(A,N)$. In view of \eqref{eq:finer}, it is also a Cauchy sequence in $\hh(\CC^m)$ and therefore converges to an entire function $f$. To show that such $f$ is in $E_w^m(A,N)$, fix $r>0$, the uniform convergence on the $N$-ball of radius $r$ shows that there exists $n_0=n_0(r)$ such that $M_N(f-f_{n_0},r)\leq 1$ so that 
\[
M_N(f,r)\exp(-Ar^\omega)\leq M_N(f_{n_0}, r)\exp(-Ar^\omega)+\exp(-Ar^\omega)\leq \|f_{n_0}\|_{A,N}^\omega +1 \leq C,
\]
the latter since the sequence $\|f_{n}\|_{A,N}^\omega$, being itself Cauchy, is bounded. Passing to the supremum over $r$, we obtain that $f\in E_w^m(A,N)$. A similar reasoning shows that the convergence of $f_n$ to $f$ holds in $E_w^m(A,N)$.
\end{proof}
Likewise, since, for any functional $\mu$ (continuous linear form) on $\hh(\CC^m)$, there exist a compact $K$ and a constant $C(K)$ such that, for all $f\in\hh(\CC^m)$, $|\mu(f)|\leq C(K)\||f||_K$, the restriction of such a functional to $E_w^m(A,N)$ is continuous with its topology of Banach space. We therefore have :
\begin{lemma}\label{th:contothernorm}
Let $\Pi$ denote a polynomial projector on $\hh(\CC^m)$. 
The restriction of $\Pi$ to $E_w^m(A,N)$ is continuous for its topology of Banach space.
\end{lemma}
We will need to know the $\|\cdot\|_{A,N}^\omega$-norm of the monomial $e_\alpha : z \to z^\alpha$ 
in $\CC^m$.
\begin{lemma}\label{th:estdeltaalpha}
We have
\[\|e_\alpha\|_{A,N}^\omega = \delta_N(\alpha) \left(\dfrac{|\alpha|}{e\omega A}\right)^{|\alpha|/\omega},\]
where $\delta_N(\alpha)=M_N(e_\alpha,1)$ is defined in \eqref{eq:deltaN}. 
\end{lemma}
\begin{proof}
By homogeneity of $e_\alpha$, we have
\[M_N(e_\alpha,t)\exp(-At^\omega)=M_N(e_\alpha, 1) t^{|\alpha|}\exp(-At^\omega)=\delta_N(\alpha)t^{|\alpha|}\exp(-At^\omega).\]
The claim follows by observing that the function $t\in \RR^+ \to t^{|\alpha|}\exp(-At^\omega)$ reaches its maximum at $t=(|\alpha|/(\omega A))^{1/\omega}$.
\end{proof}

\subsection{The convergence theorems}

As above, we work with $\CC^n=\CC^{n_1}\times \CC^{n_2}, \quad n_s\geq 1$ and let $N_s$ denote a PC-norm on $\CC^{n_s}$. 
Moreover, given $w>0$ and $A_s>0$ for $s=1,2$, we set
\[
E_s=E_{\omega}^{n_s}(A_s, N_s) \subset \hh(\CC^{n_s}), \quad s=1,2.
\] 
For the sake of simplicity, we will also write
\begin{equation}\label{eq:simplnotnorm}
\|\cdot \|_s=\|\cdot\|_{A_s,N_s}^\omega, \quad s=1,2.
\end{equation}
We first treat the case $0<\omega \leq 1$, the case $\omega \geq 1$, which is similar, is studied below. 

\begin{theorem}\label{th:convergentire} We use the notation above and assume $\omega\leq 1$. Let  $\mathcal{N}^s=(\Pi_0^s, \Pi_1^s, \dots)$ denote a Newton sequence on $\hh(\CC^{n_s})$,  $s=1,2$. Given $a_s>0$ for $s=1,2$, we define a norm $N$ on $\CC^n$ by setting
\begin{equation}\label{eq:defNormN}
N(z^1,z^2)=a_1N_1(z^1)+a_2N_2(z^2),\quad z^s\in\CC^{n_s}, \quad s=1,2.
\end{equation}
If $\mathcal{N}^s$ is converging on $E_s$ for $s=1,2$ then $\mathcal{N}^1\ntimes \mathcal{N}^2$ is converging on $E_\omega^n(A,N)$ provided that
\begin{equation}\label{eq:condonA}
 A < \min\left(\dfrac{A_1}{a_1^\omega}, \dfrac{A_2}{a_2^\omega}\right).
\end{equation} 
\end{theorem}

The result will be applied in the following form. 

\begin{corollary}\label{th:cor1conentire} Let $0<\omega\leq 1$ and $\tau_s>0$, $s=1,2$. If $\mathcal{N}^s$ is converging on the space of functions of $\omega$-type $< \tau_s$ with respect to $N_s$ in $\hh(\CC^{n_s})$, for $s=1,2$ then $\mathcal{N}^1\ntimes \mathcal{N}^2$ is converging on the space of functions of $\omega$-type $ < 1$ with respect to 
\[
N=\tau_1^{1/\omega} N_1 + \tau_2^{1/\omega} N_2.
\]
\end{corollary}

\begin{proof}
According to Remark \ref{typeandspace}, we need to show that for all $0<A<1$, $\mathcal{N}^1\ntimes \mathcal{N}^2$ is converging on $E^n_{\omega}(A,N)$. 
Let $\epsilon=1-A$ and fix $\delta > 0 $ such that $\delta/\tau_s<\epsilon$ and 
$A_s=\tau_s-\delta< \tau_s$, $s=1,2$.  Observe that $E_\omega^{n_s}(A_s, N_s)$ is included in the space of entire functions of $\omega$-type $<\tau_s$ so that $\mathcal{N}_s$ is $\hh(\CC^n)$-converging on $E_\omega^{n_s}(A_s, N_s)$ and the assumptions of Theorem 
\ref{th:convergentire} are satisfied. We now apply it with $a_s=\tau_s^{1/\omega}$, $s=1,2$ and it remains to check that Condition \eqref{eq:condonA} is satisfied. This is clear since 
\[
A=1-\epsilon  < (\tau_s-\delta)/a_s^\omega=1-\delta/\tau_s, \quad s=1,2.
\]
The theorem therefore implies convergence on $E^n_{\omega}(A,N)$ 
as required. 
\end{proof}

\begin{corollary}[Case $\tau_1=\tau_2$]\label{th:cor1conentire2} Let $0<\omega\leq 1$ and $\tau>0$. If $\mathcal{N}^s$ is converging on the space of functions of $\omega$-type $< \tau$ with respect to $N_s$ in $\hh(\CC^{n_s})$, for $s=1,2$ then $\mathcal{N}^1\ntimes \mathcal{N}^2$ is converging on the space of functions of $\omega$-type $ < \tau$ with respect to 
$N=N_1+N_2$.
\end{corollary}

\begin{proof} Use the previous corollary taking \eqref{eq:typelambdaN} into account.
\end{proof}
Note that the result is optimal if the assumption is. That is to say if there exists a function $f$ of $\omega$-type $=\tau$ with respect to $N_1$ for which $\Pi^1_d(f)$ does not converge to $f$ in $\hh(\CC^{n_1})$ then the same function regarded as a function on $\CC^{n}$ provides a function of $\omega$-type $=\tau$ with respect to $N$ for which convergence does not occur.  
\par \smallskip
A careful examination of the proof of Theorem \ref{th:convergentire} below would show that a version of the theorem holds true when the spaces $E_s$ are not defined with identical $\omega$ but rather with $\omega_s$, $s=1,2$. With, say, $\omega_1<\omega_2$,  Condition \eqref{eq:condonA} would be then changed to
\[
A<\dfrac{A_1}{a_1^\omega}\quad\text{and}\quad A < \dfrac{(e\omega_2 A_2)^{\omega_2/\omega_1}}{e\omega_1 a_2^{\omega_1}}.
\]
Such a condition is not satisfying since it does not ensure suitable convergence for all functions depending only on the group of variables $z^2$, while this convergence follows immediately from the definition of the Newton product. It seems that, when $\omega_1\neq \omega_2$, a satisfactory statement cannot be obtained without the use of a notion of partial type. We will not discuss this approach in this paper. 
\par\smallskip Given $\alpha_s\in \NN^{n_s}$, $s=1,2$, we form $\alpha=(\alpha^1,\alpha^2)\in\NN^n$ . We need to compute  $\delta_N(\alpha^1,\alpha^2)$ in terms of $\delta_{N_1}(\alpha^1)$ and $\delta_{N_2}(\alpha^2)$ where $N$ is as in \eqref{eq:defNormN}.

\begin{lemma}\label{th:estdeltaN} With the notation of the theorem, 
\[
\delta_N(\alpha)=\dfrac{|\alpha^1|^{|\alpha^1}|\alpha^2|^{|\alpha^2|}}{|\alpha|^{|\alpha|}}\;\cdot\; \dfrac{\delta_{N_1}(\alpha^1)\delta_{N_2}(\alpha^2)}{{a_1}^{|\alpha^1|}{a_2}^{|\alpha^2|}}.
\]
\end{lemma}
\begin{proof} In view of \eqref{eq:deltaN}, we have
\[
\delta_N(\alpha^1,\alpha^2)=\max\left\{|u_1|^{|\alpha^1|}|u_2|^{|\alpha^2|}\;:\; a_1N_1(u_1)+a_2N_2(u_2)=1\right\}.
\]
Setting $u_s=r_sv_s$ with $N_s(v_s)=1$, $s=1,2$, the above relation translates into
\[
\delta_N(\alpha^1,\alpha^2)=\delta_{N_1}(\alpha^1)\delta_{N_2}(\alpha^2) \max\left\{ 
r_1^{|\alpha^1|}r_2^{|\alpha^2|}\; :\; a_1r_1+a_2r_2=1
\right\}.
\]
Now, it is readily seen that the function
\[
r_1\in [0,1/a_1] \longrightarrow r_1^{|\alpha^1|}\left(\dfrac{1-a_1r_1}{a_2}\right)^{|\alpha^2|}
\]
reaches its maximum at 
\[
r_1=\dfrac{1}{a_1} \;\cdot\; \dfrac{|\alpha^1|}{|\alpha|},
\]
from wich the lemma immediately follows. 
\end{proof} 
We turn to the proof of Theorem \ref{th:convergentire}. 
\begin{proof}
Using the notation \eqref{eq:defpronewseq}, we need to show that for any $f\in E_\omega^n(A,N)$,
$\Pi_d(f)$ converges to $f$ in $\hh(\CC^n)$ as $d\to \infty$, that is, $\Pi_d(f)$ converges to $f$ uniformly on every compact subset of $\CC^n$. We fix such a $f$ and $K=K_1\times K_2$ a compact set in $\CC^n$ with $K_s\in\CC^{n_s}$ and prove that $\|f-\Pi_d(f)\|_K\to 0$ as $d \to \infty$.
\begin{step}[Banach-Steinhaus] In view of Lemma \ref{th:contothernorm}, $\Pi^s_d : E_s \to \cc(K_s)$ is a continuous linear map (for the Banach space topologies) and, from the assumption on $\mathcal{N}^s$, for all $f_s$ in $E_s$, $\Pi^s_d(f_s)$ converges to ${f_s}_{|K_s}$. Hence, by the Banach-Steinhaus Theorem, there exists a constant $\Gamma_s$ such that, see \eqref{eq:simplnotnorm} for the notation,
\begin{equation}
\|\Pi^s_d(f_s)\|_{K_s}\leq \Gamma_s\|f_s\|_s, \quad d\in \NN, \quad f_s\in E_s, \quad  s=1,2.
\end{equation}  
Likewise, using the Newton summands $\pi_d^s=\Pi^s_d-\Pi^s_{d-1}$, just as in Corollary  \ref{eq:corunfiboundedness}, 
\begin{equation}\label{eq:bsNewtonsum}
\|\pi^s_d(f_s)\|_{K_s}\leq C_s\|f_s\|_s, \quad d\in \NN, \quad f_s\in E_s, \quad  s=1,2,
\end{equation}
with $C_s=2\Gamma_s$.
\end{step}
\begin{step}[Using the strategy described in Subsection \ref{sec:strategy}] Writing $e^s_\alpha: z^s\in \CC^{n_s}\to z^{\alpha}$, we start from the power series expansion of $f$, 
\[ f(z^1,z^2)= \sum_{j=0}^\infty\sum_{|\alpha^1|+|\alpha^2|=j} c_{(\alpha^1,\alpha^2)} e_{\alpha^1}(z^1) e_{\alpha^2}(z^2).\]
A use of \eqref{eq:strat2} and Lemma \ref{th:lemmabd}  yields 
\[f-\Pi_d(f)=\sum_{j=d+1}^\infty\sum_{|\alpha^1|+|\alpha^2|=j} c_{\alpha^1\alpha^2} 
\left\{\sum_{(i_1,i_2)\in B(d,\alpha^1,\alpha^2)} \pi_{i_1}^1(e_{\alpha^1}) \pi_{i_2}^2(e_{\alpha^2}) \right\}.\]
We will denote by $R_{\alpha^1,\alpha^2}$ the sum into brackets in the above equation. 
We will prove that the series
\[%
\sum_{j=0}^\infty\sum_{|\alpha^1|+|\alpha^2|=j} \left|c_{\alpha^1\alpha^2} \right|
\|R_{\alpha^1,\alpha^2}\|_K
\]
is converging and $\|f-\Pi_d(f)\|_K$ will therefore go to $0$ as it is bounded by the remainder of a converging series. 
\end{step}
\begin{step}[Estimating $\|R_{\alpha^1,\alpha^2}\|_K$]
Using \eqref{eq:bsNewtonsum} in the above equation for $R_{\alpha^1,\alpha^2}$, we obtain 
\begin{align}
\|R_{\alpha^1,\alpha^2}\|_K & \leq C_1C_2 
\text{card}\left(B(d,\alpha^1,\alpha^2)\right) \|e_{\alpha^1}\|_1 \; \|e_{\alpha^2}\|_2 \\
\intertext{ and, using the estimate \eqref{eq:escardBd} for $\text{card}\left(B(d,\alpha^1,\alpha^2)\right)$,}
& \leq 
C_1C_2 
(|\alpha^1|+1)(|\alpha^2|+1)\|e_{\alpha^1}\|_1 \; \|e_{\alpha^2}\|_2
\end{align}
Now a use of Lemma \ref{th:estdeltaalpha} gives
\begin{multline}\label{eq:finalR}
\|R_{\alpha^1,\alpha^2}\|_K \leq C_1C_2 
(|\alpha^1|+1)(|\alpha^2|+1) \delta_{N_1}(\alpha^1)\delta_{N_2}(\alpha^2) \\
\times \left(\dfrac{|\alpha^1|}{e\omega A_1}\right)^{\frac{|\alpha^1|}{\omega}}\left(\dfrac{|\alpha^2|}{e\omega A_2}\right)^{\frac{|\alpha^2|}{\omega}}.
\end{multline}
\end{step}
\begin{step}[Estimating $\left|c_{\alpha^1\alpha^2} \right|$] Since $N_s$ is a PC-norm for $s=1,2$ so is the norm $N$ and we may therefore apply estimate \eqref{eq:estimacoefpoweseries} in Lemma \ref{th:estimacoefpoweseries} to get 
\begin{equation}%
\left|c_{\alpha^1\alpha^2} \right|  \leq \dfrac{t^{-|\alpha|}}{\delta_{N}(\alpha)}M_N(f,t)
\leq M \dfrac{t^{-|\alpha|}}{\delta_{N}(\alpha)} e^{tA}, \quad \alpha=(\alpha^1,\alpha^2),\quad t\geq 0,
\end{equation}%
where we used $f\in E_\omega(A,N)$. Since, as in the proof of Lemma \ref{th:estdeltaalpha},  the function $t\to t^{-|\alpha|}e^{t^\omega A}$ reaches its minimum for $t=(|\alpha|/(\omega A))^{1/\omega}$, we have
\[%
\left|c_{\alpha^1\alpha^2} \right| \leq
M \left(\dfrac{|\alpha|}{e\omega A}\right)^{-|\alpha|/\omega}\cdot\dfrac{1}{\delta_N(\alpha)}. 
\]%
Next, we use Lemma \ref{th:estdeltaN} to handle the term $\delta_{N}(\alpha)$ and, after some simple calculation, we arrive to
\begin{equation}\label{eq:finalcoef}
\left|c_{\alpha^1\alpha^2} \right| \leq M
\left(\dfrac{1}{e\omega A}\right)^{-|\alpha|/\omega}\cdot |\alpha|^{|\alpha|-1/\omega}
\cdot \dfrac{{a_1}^{|\alpha^1|}{a_2}^{|\alpha^2|}}{|\alpha^1|^{|\alpha^1|}|\alpha^2|^{|\alpha^2|}\delta_{N_1}(\alpha^1)\delta_{N_2}(\alpha^2)}.
\end{equation}
\end{step}
\begin{step}[Conclusion] Putting \eqref{eq:finalR} and \eqref{eq:finalcoef} together, setting $C=MC_1C_2$, we obtain 
\begin{multline*}
\left|c_{\alpha^1\alpha^2} \right|  \|R_{\alpha^1,\alpha^2}\|_K  \leq C (|\alpha^1|+1)
(|\alpha^2|+1) \\ \cdot \left(\dfrac{a_1^{\omega}A}{A_1}\right)^{|\alpha^1|/\omega}
\left(\dfrac{a_2^{\omega}A}{A_2}\right)^{|\alpha^2|/\omega} 
\times \left\{\dfrac{|\alpha^1|^{|\alpha^1|}|\alpha^2|^{|\alpha^2|}}{|\alpha|^{|\alpha|}}\right\}^{\frac{1}{\omega}-1}.
\end{multline*}
The terms into brackets is obviously smaller than one
(since $|\alpha^s|\leq |\alpha|$), hence we have 
\[
\left|c_{\alpha^1\alpha^2} \right|  \|R_{\alpha^1,\alpha^2}\|_K  \leq M (|\alpha|+1)^2 q_1^{|\alpha^1|}q_2^{|\alpha^2|}\quad\text{with}\quad q_s=\left(\dfrac{a_s^{1/\omega}A}{A_s}\right)^{1/\omega} <1,
\]
the latter by assumption \eqref{eq:condonA}. The convergence of the series of general term $c_{\alpha^1\alpha^2} R_{\alpha^1,\alpha^2}$ follows and according to the second step, this concludes the proof of the theorem.  
\end{step}
\end{proof}
We will now deal with the case $\omega \geq 1$. 
\begin{theorem}\label{th:convergentire2} We use the notation above and assume $\omega \geq 1$. Let  $\mathcal{N}^s=(\Pi_0^s, \Pi_1^s, \dots)$ denotes a Newton sequence on $\hh(\CC^{n_s})$,  $s=1,2$. Given $a_s>0$ for $s=1,2$, we define a norm $N$ on $\CC^n$ by
\begin{equation}\label{eq:defNormN2}
N(z^1,z^2)=\left(a_1N_1^\omega(z^1)+a_2N_2^\omega(z^2)\right)^{1/\omega},\quad z^s\in\CC^{n_s}, \quad s=1,2.
\end{equation}
If $\mathcal{N}^s$ is converging on $E_s$ for $s=1,2$ then $\mathcal{N}^1\ntimes \mathcal{N}^2$ is converging on $E_\omega^n(A,N)$ provided that
\begin{equation}\label{eq:condonA2}
 A < \min\left(\dfrac{A_1}{a_1}, \dfrac{A_2}{a_2}\right).
\end{equation} 
\end{theorem}

Thus the mains changes are the definition of the norm $N$ and the assumptions on $A$ in which $\omega$ disappear. It is worth noting that the definition of $N$ and the conditions on $A$ in both theorems are continuous with respect to $\omega$ around 
$\omega=1$. Note also that the condition $\omega \geq 1$ ensures that equation 
\eqref{eq:defNormN2} defines a norm. It is obviously a PC-norm as soon as the $N_s$ are.
\par\smallskip
\par Corollaries \ref{th:cor1conentire} and \ref{th:cor1conentire2} remain true with obvious changes. We will not state them. 
\par The change in the definition on $N$ requires a modification of Lemma \ref{th:estdeltaN}.

\begin{lemma}\label{th:estdeltaN2} With the notation of Theorem \ref{th:convergentire2}, in particular with $N$ as in \eqref{eq:defNormN2}, we have 
\[
\delta_N(\alpha)=\left(\dfrac{|\alpha^1|^{|\alpha^1}|\alpha^2|^{|\alpha^2|}}{
|\alpha|^{|\alpha|}{a_1}^{|\alpha^1|}{a_2}^{|\alpha^2|}}\right)^{1/ \omega}\;\cdot\;\delta_{N_1}(\alpha^1)\delta_{N_2}(\alpha^2)
\]
\end{lemma}

\begin{proof} Working as in the proof of Lemma \ref{th:estdeltaN}, we find
\[
\delta_N(\alpha^1,\alpha^2)=\delta_{N_1}(\alpha^1)\delta_{N_2}(\alpha^2) \max\left\{ 
r_1^{|\alpha^1|/\omega}r_2^{|\alpha^2|\omega}\; :\; a_1r_1+a_2r_2=1
\right\},
\]
but this is only the $\omega$-th root of the function in the proof of 
Lemma \ref{th:estdeltaN2} so that the result immediately follows. 
\end{proof}
\begin{proof}[Proof of Theorem \ref{th:convergentire2}] The proof is identical to that of Theorem \ref{th:convergentire} up to the conclusion of Step 4 in which we use the estimate of Lemma \ref{th:estdeltaN2} instead of that of Lemma \ref{th:estdeltaN}. The corresponding estimate in the conclusion even simplifies to become 
\[
\left|c_{\alpha^1\alpha^2} \right|  \|R_{\alpha^1,\alpha^2}\|_K  \leq M (|\alpha^1|+1)
(|\alpha^2|+1) \cdot \left(\dfrac{a_1 A}{A_1}\right)^{|\alpha^1|/\omega}
\left(\dfrac{a_2 A}{A_2}\right)^{|\alpha^2|/\omega},
\]
from which the conclusion is immediate. 
\end{proof}

\begin{remark}[About the Assumption on the norm $N$.] \hfill\par
\begin{asparaenum}[(A)]
\item Any norm $N_s$ is dominated by the PC-norm $\mathbf{N}_s$ defined by \[\mathbf{N_s}(z)=\max\{N(\lambda_1z_1,\dots, \lambda_nz_n),\; |\lambda_i|=1, \; i=1,\dots, n\},\]
and $\mathbf{N}_s$ is the smallest norm satisfying the property. Since 
\[E_\omega^{n_s}(A_s,\mathbf{N}_s)\subset E_\omega^{n_s}(A,N_s),\] we may apply Theorem \ref{th:convergentire} to get a space depending on $\mathbf{N}$ (defined from the norms $\mathbf{N}_s$ as in the above theorems) on which the Newton product converges. This space however is unlikely to be optimal, as follows from the next remark. 
\item In the case of entire functions of exponential type, another approach is available that enable to eliminate the assumption on the norm $N$ to be $PC$. By representing such a function $f$ (of exponential type $\tau<1$ with respect to $N$) by a Laplace transform of an analytic functional,
see \cite[Section 4.5]{hormander}, we can write $f$ as
 \[f(w)=\int_{\{N^\star(z)\leq \tau\}} \exp \langle z, w \rangle d\mu (z),\]
where $\mu$ is a complex measure supported on the ball $\{N^\star(z)\leq \tau\}$ and $N^\star$ is the \textbf{dual norm} of $N$ (without assuming that $N$ is PC). Recall that
\[N^\star(z)= \max\big\{|\langle z,\xi\rangle|, \; N(\xi)=1\big\}.\]
 This yields a new way of computing $\Pi(f)$, hence $f-\Pi(f)$, namely 
\[\Pi(f)(w)=\int_{\{N^\star(z)\leq \tau\}} \Pi\big(\exp \langle z, \cdot \rangle\big)(w) d\mu (z),\]
which is interesting since $w\to\exp \langle z, w \rangle$ is the product $\exp \langle z^1,
 w^1 \rangle\times \exp \langle z^2, w^2 \rangle$ so that we may use the available formula for the Newton projector of a product function (Theorem \ref{th:fbsnppppart2}). Let us just point out how the formula for $N$ naturally comes from the norms $N_s$ through duality. In fact, if $N=\tau_1N_1+ \tau_2N_2$ as in Corollary \ref{th:cor1conentire} then one readily checks that 
\[N^\star(z^1,z^2)= \max\left\{(\tau_1N_1)^\star(z^1), (\tau_2N_2)^\star(z^2) \right\} = \max\left\{\frac{N_1^\star(z^1)}{\tau_1}, \frac{N_2^\star(z^2)}{\tau_2}\right\}.\]
Hence if $F$ is an entire function of exponential type $< 1$ with respect to $N$, for some measure $\mu$ we have 
\begin{align*}f(w)&= \int_{\{N^\star(z)\leq 1\}} \exp \langle z, w \rangle d\mu(z^1,z^2)\\
&= \iint_{\{N_1^\star(z_1)\leq \tau_1\}\times\{N_2^\star(z_2)\leq \tau_2\}} \exp\langle z^1, w^1 \rangle\exp\langle z^2, w^2 \rangle d\mu(z^1,z^2).
\end{align*}
And the convergence of $\Pi_d(f)$ to $f$ can be obtained as from the convergence of 
$\Pi_d$ for the product function $w\to \exp\langle z^1, w^1 \rangle \times \exp\langle z^2, w^2 \rangle$ (the convergence being uniform in $w^s$). We omit the details.  
\end{asparaenum}
\end{remark}

\subsection{Example}
We illustrate the above convergence theorems in the case 
of the product of two Newton sequences of Kergin interpolation projectors as in Subsection
\ref{sec:exampleHK} for which deep approximation results are available for entire functions. Let us first recall such a result. We use the notation introduced above. Give $\omega>0$, we set 
\begin{equation}\label{eq:defcconstant} c=c(\omega)= \int_{0}^{1/2} \dfrac{t^{\omega-1}}{1-t}dt.\end{equation}
In particular, \[c(1)=\ln 2.\]
\par Given a sequence of points $a^s_d$ in $\CC^{n_s}$, $d\in \NN$, such that the sequence of their norms $N_s(a^s_d)$ is non decreasing, we define the counting function $\mathcal{N}_s(r)$ as the number of interpolation points whose $N^s$-norm is not bigger than $r$, that is 
\[\mathcal{N}_s(r)= \text{card}\{i\in \NN\;:\; N^s(a^s_i)\leq r\}.\]
The $\omega$-\textbf{density} $\Delta_s$, with respect to $N_s$, of the sequence $a^s_d$ is then defined as 
\[\Delta_s= \liminf_{r\to\infty} \dfrac{\mathcal{N}_s(r)}{r^\omega}.\]
It is known that, for all entire function of $\omega$-type $\tau^s$ such that 
$\tau_s<c \Delta_s$, the sequence of Kergin interpolation polynomials $K[a_0^s,\dots,a_d^s;f]$ converges uniformly to $f$ on every compact subset of $\CC^{n_s}$. This is a multivariate generalization of a Theorem of Gelfond \cite{guelfond}. It was first proved in the case of the standard euclidean norm by Bloom \cite{bloduke} and then extended to an arbitrary norm by Andersson and Passare \cite{pasjat}. 
\par
Here, a direct application of Corollary \ref{th:cor1conentire2} yields the following. 
\begin{theorem}\label{th:appltiker} Let $(a^s_d)$, $s=1,2$ be two sequences of points as above.  If their density are equal, i.e.  $\Delta_1=\Delta_2 (=\Delta)$, then, for all entire functions on $\CC^n=\CC^{n_1}\times\CC^{n_2}$ of $\omega$-type $\tau$ with respect to the norm $N$, see below, satisfying 
\[\tau < \Delta c,\]
where $c$ is defined in \eqref{eq:defcconstant}, then the sequence of polynomials \[(\kg[a_0^1,\dots,a_d^1]) \ntimes (\kg[a_0^2,\dots,a_d^2])(f)\] converges uniformly to $f$ on every compact subset of $\CC^n$.  
Here $N=N_1+N_2$ if $\omega \leq 1$ and $N=(N_1^\omega+N_2^\omega)^{1/\omega}$ if $\omega\geq 1$. 
\end{theorem}
Of course, the case $\Delta_1\neq\Delta_2$ is handled by Corollary
\ref{th:cor1conentire}.
%
%
%%%%%%%%%%%%%%%%%%%%%%%%%%%%%%%%%%%%%%%%%%%%%%%%%%%%%%%%%%%%%%%%%%%%%%%%%%%%%%%%%%%
\section{Spaces of smooth functions}\label{sec:smooth}
\subsection{Adapting the tools} Roughly the same technique as in Section \ref{sec:holomorphic} can be used to derive results on spaces of differentiable functions. We will omit some details of the proofs where they are similar to those previously given. The results obtained are not optimal, this will be explained below. Although we will present the results in a more general setting, the reader may assume that, in what follows, all the compact sets considered are convex bodies (compact convex sets of non empty interior). Given a fat compact set $\kappa$ ($\kappa$ is the closure of its interior), we denote by $\cc^m(\kappa)$ the space of all functions which are $m$-times continuously differentiable on the interior of $\kappa$ and whose all derivatives of order $\leq m$ extend continuously to $\kappa$. It is a Banach space when endowed with the norm
\begin{equation}
\|f\|_{m,\kappa}=\max_{|\alpha|\leq m} \|D^\alpha f\|_\kappa.
\end{equation} 
First, we need a stronger notion of Bernstein-Markov measure, see Subsection \ref{sec:BMmeasures}. Let $\mu$ be a probability measure on $\kappa$. If there exists $\theta >0$ such that, for some constant $C_\mu$, we have
\begin{equation}\label{eq:strongBM}
\|p\|_\kappa \leq C_\mu (\deg p)^{\theta} \| p\|_2, \quad p\in \pol(\RR^n),
\end{equation}
where as usual $\| p\|^2_2=\int_\kappa p^2(x) d\mu(x)$,
we say that $\mu$ is a $\theta$-\textbf{strong Bernstein-Markov measure} (\textbf{SBM}) on $\kappa$. When there exists such a measure on $\kappa$, we say that $\kappa$ is a $\theta$-SBM compact.
Observe that the sub-exponential term $(1+\varepsilon)^{\deg (p)}$ in \eqref{eq:inBM} is replaced in \eqref{eq:strongBM} by a polynomial term in the degree. 
The classical Nikolskii inequality states that $dx$ (Lebesgue measure) is $1$-SBM for $[-1,1]$, see \cite[Theorem 3.1.4]{rassias}. Zeriahi showed in \cite{zer2} that the Lebesgue measure is SBM for a large class of compact sets. From the bounds in \cite[p. 686]{zer2}, we deduce that, when $K$ is a convex body in $\RR^n$, then $\theta$ can be taken as $3n/2$. Such a $\theta$ is probably not optimal but we are not aware of more precise bounds. We omit the proof of the following lemma which is similar to that of [8, Lemma 2, p. 290], see also Subsection \ref{sec:BMmeasures}. 

\begin{lemma} If $\mu_s$ is a $\theta_s$-SBM measure on $K_s\in \RR^{n_s}$ for $s=1,2$ then $\mu_1\times \mu_2$ is a $(\theta_1+\theta_2)$-SBM measure on $K_1\times K_2$. 
\end{lemma}
Recall now that a compact $\kappa$ satisfies a \textbf{Markov inequality} of \textbf{exponent} $r$ if, for some positive constant $M_\kappa$, we have
\begin{equation}
\label{eq:markovineq}  
\|D^\alpha p\|_\kappa \leq M_\kappa (\deg p)^{r|\alpha|} \|P\|_\kappa, \quad\alpha\in\NN^n, \quad p\in \pol(\RR^n).
\end{equation}
We refer to \cite{plesniak4} for a large class of compact sets satisfying a Markov inequality, see also the surveys \cite{plesniak2,plesniak3}. In the case of a convex body, the exponent $r$ can be taken as $2$ (in any dimension), see \cite{wilhelmsen}.
\par It is readily seen that if $\kappa_s$ satisfies a Markov inequality of exponent $r_s$, $s=1,2$, then $\kappa_1\times \kappa_2$ satisfies a Markov inequality of exponent $r=\max(r_1,r_2)$. 
\par
If $\kappa$ is a $\theta$-SBM compact satisfying a Markov inequality of exponent $r$, then there exists a constant $C_{\mu, \kappa}=C(\mu)M(\kappa)$, such that for all derivatives $D^\alpha$ we have
\begin{equation}
\label{eq:difstrongBM}
\|D^\alpha p\|_\kappa\leq C_{\mu, \kappa} (\deg p)^{\theta + r|\alpha|} \| p\|_2, \quad p\in \pol(\RR^n).
\end{equation}
Finally, we say that $\kappa$ is a \textbf{Jackson compact set} if for all function $f$ in $\cc^{m}(\kappa)$, we have 
\begin{equation}\text{dist}(f, \pol_d(\RR^n)) \leq M_J d^{-m}\end{equation} where the distance is with respect to the uniform norm on $\kappa$ and $M_J$ depends only on $f$, $m$ and $\kappa$. For this notion of Jackson sets we refer to \cite{plesniak} and the references therein.  In view of \cite[Theorem 2]{bagby} if $\kappa$ is quasi-convex (i.e. satisfies the Whitney property P), so that, by the Whitney extension Theorem \cite{whitney, bierstone}, all $f$ in $\cc^m(\kappa)$ extends to a function $\cc^m$ on a neighborhood of $\kappa$ then for all degree $d$,  there exists a polynomial $t_d$ (of near-best simultaneous approximation of $f$) such that 
\begin{equation}\label{eq:simappr}
\|D^\alpha (f-t_d)\|_\kappa \leq \mathbf{M}_J /d^{m-|\alpha|}, \quad d\in \NN, 
\end{equation}
where $\mathbf{M}_J=M_J(m,f)\geq M_J$ depends only on $f$ and $m$ and $\kappa$. Such sets, in particular, convex bodies, are therefore Jackson sets. 
\par
Given a $\theta$-SBM measure $\mu$ on $\kappa$, we may construct the sequence of orthonormal polynomials $b_\alpha$ and the corresponding orthogonal projection 
\[\pro_{d,\mu}(f)= \sum_{|\alpha|\leq d} c_\alpha(f) b_\alpha,\quad c_\alpha(f) = \langle f, b_\alpha\rangle =\int_K f(x){b}_\alpha(x)d\mu(x),\]
as in Subsection \ref{sec:BMmeasures}.
\begin{lemma}\label{th:prodsbm} Let $K\subset \RR^n$ be a $\theta$-SBM compact set with measure $\mu$ as well as a Jackson and Markov compact set with exponent $r$. For all $f\in \cc^{\mathbf{m}}(K)$, we have
\begin{equation}\label{eq:conforcp} f = \sum_{d=0}^\infty \sum_{|\alpha|=d}  c_\alpha(f) b_\alpha\quad\text{in \; ${\cc}^{\overline{m}}(K)$,}\end{equation}
whenever $\mathbf{m}> \theta + r\overline{m}+1$. 
\end{lemma}
\begin{proof} We denote by $\mu$ the $\theta$-SBM measure on $K$. The reasoning is similar to that given in Subsection \ref{sec:BMmeasures}, see also \cite{zer2}. Let us write $H_d = \sum_{|\alpha|=d}  c_\alpha(f) b_\alpha$. We prove that $\sum_{d=0}^\infty D^\beta H_d$ uniformly converges on $K$ for $\|\alpha\leq \overline{m}$.  To explain the condition on $\mathbf{m}$, let us just observe that, the assumptions on $K$ enable us to use \eqref{eq:difstrongBM} so that
\begin{equation}\label{eq:estHd}
\|D^\beta H_d\|_K \leq C_{\mu,K} d^{\theta+r\overline{m}}\|H_d\|_2,
\end{equation}
Now, let us just observe that here, calling $t_{d-1}$ a best uniform approximation polynomial of $f$ in $\pol_{d-1}(\RR^n)$, see \eqref{eq:simappr},
we have
\begin{align}
\|H_d\|_2 &=\|\pro_{d,\mu}(f)-\pro_{d-1,\mu}(f)\|_2 & 
\\
& \leq \|f-\pro_{d,\mu}(f)\|_2+ \|f-\pro_{d-1,\mu}(f)\|_2 & 
 \\ 
& \leq 2 \|f-\pro_{d-1,\mu}(f)\|_2 & \text{($\pro_{d,\mu}(f)$ is best in $\pol_d$)}
\\ 
&\leq 2 \|f-t_{d-1}\|_2 & \text{($\pro_{d-1,\mu}(f)$ is best in $\pol_{d-1}$)}\\ 
\label{eq:estcalphabetasmooth}
&  \leq 2\|f-t_{d-1}\|_K \leq \frac{2M_J}{(d-1)^{\mathbf{m}}} \leq \frac{M'}{d^{\mathbf{m}}}& \text{($K$ is Jackson and $f\in {\cc}^{\mathbf{m}}(K)$)}.
\end{align}
This bound together with inequality \eqref{eq:estHd} now yields
\[
\|D^\beta H_d\|_K \leq C_{\mu,K} M' d^{\theta+r\overline{m}-\mathbf{m}}, 
\]
so that the series $\sum_{d=0}^\infty D^\beta H_d$ converges normally on $K$ as soon as 
$\theta+r\overline{m}-\mathbf{m}<-1$. 
This shows that, under the assumption on $\mathbf{m}$, the right hand side of \eqref{eq:conforcp} converges in $\cc^{\overline{m}}(K)$ to a certain function $g$. Since 
the convergence of the series to $f$ plainly holds in $L^2$, $g$ coincides almost everywhere with $f$ on $K$, hence everywhere by continuity. 
\end{proof}

\subsection{The convergence theorem}
It is important to note the three levels of differentiability $\mathbf{m}$, $\overline{m}$ and $m$ that occur in the following statement. 
\begin{theorem}\label{th:convdiffonc} For $s=1,2$, we let $K_s$ denote a compact set in  $\RR^{n_s}$ such that
\begin{enumerate}
\item $K_s$ is a $\theta_s$-SBM compact,
\item $K_s$ satisfies a Markov inequality of exponent $r_s$. 
\end{enumerate}
We assume further that $K_1\times K_2$ is a Jackson compact set in $\RR^n$, $n=n_1+n_2$ and $\kappa_s$ is a compact subset of $K_s$. We set $K=K_1\times K_2$ and $\kappa=\kappa_1\times \kappa_2$. 
\par 
Let $\mathcal{N}^s=(\Pi_0^s, \Pi_1^s, \dots)$ denote a Newton sequence on $\cc^{\overline{m}}(K_s)$,  $s=1,2$.
\par
If $\mathcal{N}^s$ is $\cc(\kappa_s)$-converging on $\cc^m(K_s)$ for $s=1,2$, $m\geq \overline{m}$, then $\mathcal{N}^1\ntimes \mathcal{N}^2=(\Pi_0, \Pi_1, \dots)$ is $\cc(\kappa)$-converging on $\cc^\mathbf{m}(K)$ provided that
\begin{equation}\label{eq:condonp}
\mathbf{m}> 3/2+(\theta_1+\theta_2)+m(r_1+r_2).\end{equation} 
\end{theorem}
In the above statement, it is implicit that when $\Pi^s_d$, $s=1,2$, is a continuous projector on 
$\cc^{\overline{m}}(K_s)$ then $\Pi^1_d\ntimes\Pi^2_d$ is a well defined continuous projector on $\cc^{\overline{m}}(K)$. Such a result is not established in \cite{bercal1} where the space considered are the usual Fréchet space $\cc^m(\Omega_s)$ where $\Omega_s$ is open in $\RR^{n_s}$. The proof in the present case is similar, we omit it. The reader may freely modify the hypothesis in the statement above assuming that all the projectors are defined on the space of functions differentiable on the whole space (with the same level of differentiability) and extend continuously to the spaces indicated, see the statement of Corollary \ref{th:convdiffoncconvex2} below. This is plainly the case in the examples presented in Subsection \ref{secexadiff}. 
\begin{corollary}
\label{th:convdiffoncconvex}
For $s=1,2$, we let $K_s$ denote a convex body in  $\RR^{n_s}$. We set $K=K_1\times K_2$ and $\kappa=\kappa_1\times \kappa_2$. 
\par 
Let $\mathcal{N}^s=(\Pi_0^s, \Pi_1^s, \dots)$ denote a Newton sequence on $\cc^{\overline{m}}(K_s)$,  $s=1,2$.
\par
If $\mathcal{N}^s$ is $\cc(\kappa_s)$-converging on $\cc^m(K_s)$ for $s=1,2$, $m\geq \overline{m}$, then $\mathcal{N}^1\ntimes \mathcal{N}^2=(\Pi_0, \Pi_1, \dots)$ is $\cc(\kappa)$-converging on $\cc^\mathbf{m}(K)$ provided that

\begin{equation}\label{eq:condonpconv}
\mathbf{m}> 3n/2+4m+3/2.
\end{equation}
\end{corollary}
\begin{proof} Convex bodies satisfies all the requirements with $\theta_s=3n_s/2$ and $r_s=2$. 
\end{proof}

We will apply the result in the following form. 

\begin{corollary}\label{th:convdiffoncconvex2} For $s=1,2$, we let $\kappa_s$ denote a convex body  in  $\RR^{n_s}$. We set $\kappa=\kappa_1\times \kappa_2$. 
\par 
Let $\mathcal{N}^s=(\Pi_0^s, \Pi_1^s, \dots)$ denote a Newton sequence on $\cc^{\overline{m}}(\RR^n)$,  $s=1,2$. We assume that all the projectors as well as those of $\mathcal{N}^1\ntimes \mathcal{N}^2=(\Pi_0, \Pi_1, \dots)$ extend continuously to, respectively, $\cc^{\overline{m}}(\kappa_s)$ and $\cc^{\overline{m}}(\kappa)$.
\par
If $\mathcal{N}^s$ is $\cc(\kappa_s)$-converging on the space of function $m$ times continuous differentiable on a neighborhood of $\kappa_s$ for $s=1,2$, $m\geq \overline{m}$, then $\mathcal{N}^1\ntimes \mathcal{N}^2=(\Pi_0, \Pi_1, \dots)$ is $\cc(\kappa)$-converging on the space of all $\mathbf{m}$ differentiable functions on a neighborhood of $\kappa$ provided that 
\begin{equation}\label{eq:condonpconv2}
\mathbf{m}> 3n/2+4m+3/2.
\end{equation}
\end{corollary}
\begin{proof}  Apply the previous corollary with $K_s$, running in a basis of neighborhood of $\kappa$ (formed of convex bodies). 
\end{proof}
Let us point out the limitation of the above theorem and its corollaries. The assumption on $\mathbf{m}$ clearly depends on the method of proof. It seems natural to expect that the theorem holds with $\mathbf{m}=m$ and, if this is not true, it would ne interesting to explain the reason why the Newton product procedure induces a loss.  

\begin{proof}[Proof of Theorem \ref{th:convdiffonc}] We use the strategy described in Subsection \ref{sec:strategy} starting from the Fourier expansion \eqref{eq:conforcp} on $K_1\times K_2$ with $\mu=\mu_1\times \mu_2$ where $\mu_s$ is a $\theta_s$-SBM on
 $K_s$, so that $\mu$ is $\theta$-SBM on $K=K_1\times K_2$ in view of Lemma \ref{th:prodsbm}. Besides $K$ satisfies a Markow inequality of exponent $\max (r_1,r_2)$. Hence, since the assumption on $\mathbf{m}$ implies
\begin{equation} 
\mathbf{m}> 1 + (\theta_1+\theta_2)+ \overline{m}\max(r_1,r_2) , 
\end{equation}
in view of Lemma \ref{th:prodsbm}, we have
\begin{equation}
\label{eq:profccK}%
 f=\sum_{j=0}^\infty \sum_{|\alpha|+|\beta|=j} c_{\alpha\beta}(f) b_{\alpha,\beta},\quad \text{in ${\cc}^{\overline{m}}(K)$}.
\end{equation}
 As in the proof of Theorem \ref{th:converghK}, the assumption on the convergence together with the uniform boundedness principle on the Banach space $\cc^m(K_s)$ provide us with a positive constant $C_s$ such that (not confuse the projector $\Pi^s_d$ with its Newton summand $\pi_d^s$) :
\begin{equation}\label{eq:profccK2}
\|\pi^s_d(f_s)\|_\kappa \leq C_s \|f_s\|_{m,K_s}, \quad f_s\in {\cc}^{m}(K_s), \quad s=1,2. 
\end{equation}
Now, since the convergence in \eqref{eq:profccK} holds in ${\cc}^{\overline{m}}(K)$ on which $\Pi_d$ is continuous then we  may permute $\Pi_d$ with the sum in the series expansion to obtain (see \eqref{eq:profhK_2})
\begin{multline}\label{eq:profccK3}
f-\Pi_d(f) = 
\sum_{j=d+1}^\infty \sum_{|\alpha|+|\beta|=j}c_{\alpha\beta}(f)
\sum_{(i_1,i_2)\in B(d,\alpha,\beta)} \pi^1_{i_1}\big(b (\alpha, \mu_1,\cdot)
\big)\pi^2_{i_2} \big(b (\beta, \mu_2, \cdot)\big).
\end{multline}
At this point, we could continue the proof as in that of Theorem \ref{th:converghK}. Yet the term in \eqref{eq:profhK_5} which is innocuous in the case of holomorphic functions (because it is dominated by a geometric sequence) should be taken into account and would lead to a weaker estimate. To avoid this term, we will use a somewhat more tricky argument. 
\par
Let us denote by $R_j$ the $j$-term in \eqref{eq:profccK3}. Permuting the sums, we obtain
\begin{equation}\label{eq:proofccK11}
H_j= \sum_{i_2=0}^j \left(\sum_{i_1=d+1-i_2}^j \sum_{(\alpha,\beta)\in \mathbf{B}(i_1,i_2,j)} c_{\alpha\beta}(f) \pi^1_{i_1}\big(b (\alpha, \mu_1,\cdot) \pi^2_{i_2}\big(b (\beta, \mu_2,\cdot)\right),
\end{equation}
where
\[%
\mathbf{B}(i_1,i_2,j)=\big\{(\alpha,\beta)\in \NN^{n_1}\times \NN^{n_2} : |\alpha|\geq i_1, \; |\beta|\geq i_2, \; |\alpha|+ |\beta| =j\big\}.
\]%
Fix $z^1 \in K_1$. Let us concentrate on the term between brackets in \eqref{eq:proofccK11}. Since $\pi_{i_2}$ is linear, it can be seen as 
\[\pi_{i_2}(h), \quad h(\cdot)= \sum_{i_1=d+1-i_2}^j \sum_{(\alpha,\beta)\in \mathbf{B}(i_1,i_2,d)} c_{\alpha\beta}(f) \pi_{i_1}^1(b (\alpha, \mu_1, \cdot))(z^1) b (\beta, \mu_2, \cdot).\]
Observe that $h=h_{z^1}$ is a polynomial (in $z^2$) of degree $\leq j$ and, since the $b (\alpha, \mu_2, z^2)$ are orthonormal,  we have
\begin{equation}\label{eq:proofccK12}
\|h\|_{2}= \sqrt{\sum_{i_1=d+1-i_2}^j \sum_{(\alpha,\beta)\in \mathbf{B}(i_1,i_2,j)} c^2_{\alpha\beta}(f) \times \big(\pi^1_{i_1}(b (\alpha, \mu_1,\cdot) (z^1)\big)^2}.
\end{equation}
Now, 
\begin{align}
 \|\pi^2_{i_2}(h)\|_{\kappa_2} 
&\leq C_2 \|h\|_{m, K_2}, & \text{by \eqref{eq:profccK2}} \\
\label{eq:proofccK13} &\leq  C_2C_{\mu_2,K_2} j^{\theta_2+mr_2} \|h_{z^1}\|_{2}, & \text{by \eqref{eq:difstrongBM}, since $h\in \pol_j(\RR^{n_2})$}.
\end{align}
At this point, we have
\begin{equation}\label{eq:proofccK14}\max_{z^2\in \kappa_2} |H_j(z^2,z^1)|\leq \sum_{i_2=0}^j C_2 C_{\mu_2,K_2} j^{\theta_2+mr_2} \|h_{z^1}\|_{2}, \quad \text{all $z^1\in \kappa_1$.}\end{equation}
Yet, with the same reasoning as above, 
\begin{align} \left\|\pi^1_{i_1}\big(b (\alpha, \mu_1,\cdot)\big)\right\|_{\kappa_1} 
&\leq C_1 \|b (\alpha, \mu_1,\cdot)\|_{m, K_1}, & \text{by \eqref{eq:profccK2}} \\
&\leq  C_1 C_{\mu_1,K_1} |\alpha|^{\theta_1+mr_1} \|b (\alpha, \mu_1,\cdot)\|_{2}, & \text{by \eqref{eq:difstrongBM}}\\
&\leq   C_1 C_{\mu_1,K_1} j^{\theta_1+mr_1}, & \text{by normality.}
\end{align}
Using this estimate in the bound for $\|h_{z^1}\|_{2}$ in \eqref{eq:proofccK13}, we arrive at 
\[\max_{z^1\in \kappa_1} \|h_{z^1}\|_{2} \leq C_1 C_{\mu_1,K_1} j^{\theta_1+mr_1}
\big(\sum_{i_1=d+1-i_2}^j \sum_{(\alpha,\beta)\in \mathbf{B}(i_1,i_2,j)} c^2_{\alpha\beta}(f)\big)^{1/2}.\]
Returning to $H_j$ in \eqref{eq:proofccK14}, we now have 
\[\|H_j\|_\kappa \leq C_2C_{\mu_2,K_2}C_1 C_{\mu_1,K_1} j^{\theta_1+mr_1+\theta_2+mr_2}
\sum_{i_2=0}^j \Big(\sum_{i_1=d+1-i_2}^j \sum_{(\alpha,\beta)\in \mathbf{B}(i_1,i_2,j)} c^2_{\alpha\beta}(f)\Big)^{1/2}.\]
To deal with the right hand term, we use the concavity inequality for the square root which reads as
\[\sum_{i_2=0}^j \frac{\sqrt{\square_{i_2}}}{j+1} \leq \sqrt{\frac{1}{j+1} \sum_{i_2=0}^j \square_{i_2}}\]
which gives the first line of the following (where we shortens the notation for clarity)
\begin{align}
\sum_{i_2=0}^j \Big(\sum_{i_1=\dots}^j \sum_{(\alpha,\beta)\in\dots} c^2_{\alpha\beta}(f)\Big)^{1/2} &\leq \sqrt{j+1} \Big(\sum_{i_2=0}^j \sum_{i_1=\dots}^j \sum_{(\alpha,\beta)\in \dots} c^2_{\alpha\beta}(f)\Big)^{1/2}\\
&= \sqrt{j+1}\Big(\sum_{|\alpha|+|\beta|=j} c^2_{\alpha\beta}(f)\Big)^{1/2} \\
&=\sqrt{j+1}\|\pro_{j,\mu}(f)-\pro_{j-1,\mu}(f)\|_2\\
\intertext{and, in view of \eqref{eq:estcalphabetasmooth} taking into account that $K$ is Jackson and $f\in{\cc}^{\mathbf{m}}(K)$,}
& \leq C'\sqrt{j+1}j^{-\mathbf{m}} \leq C''j^{1/2-\mathbf{m}}. 
\end{align}
So the final estimate for $\|H_j\|_\kappa $ is
\[\|H_j\|_\kappa \leq C''C_2C_{\mu_2,K_2}C_1 C_{\mu_1,K_1} j^{1/2+\theta_1+mr_1+\theta_2+mr_2-\mathbf{m}}.\]
Since, see \eqref{eq:profccK3}, we have 
\[\|f-\Pi_d(f)\|_\kappa \leq \sum_{j=d+1}^\infty \|H_j\|_\kappa;\]
the uniform convergence follows as soon as $1/2+\theta_1+mr_1+\theta_2+mr_2-\mathbf{m}<-1$ which is the assumption on $\mathbf{m}$. 

\end{proof}
\begin{remark} Although we consider only $\cc(\kappa)$-convergence (uniform convergence on $K$), the proof works with slight modification for $\cc^{\overline{m}}$-convergence (assuming of course that such a convergence holds for the partial sequences). In fact, in \eqref{eq:profccK3}, we would have to estimate $\|D^\gamma f-D^\gamma\Pi_d(f)\|_\kappa$ for $|\gamma|\leq \overline{m}$ and we would just just have to compute the $D^\gamma$ derivative of the products $\pi^1_{i_1}\big(b (\alpha, \mu_1,\cdot)\big)\times \pi^2_{i_2} \big(b (\beta, \mu_2, \cdot)\big)$ with the help of the Leibniz formula. Clearly, this would only further increases the acceptable value for $\mathbf{m}$. In particular, the reasoning would lead to : 
\end{remark}

\begin{corollary}[To the proof of Theorem \ref{th:convdiffonc}] With the same assumptions of Theorem \ref{th:convdiffonc} on the compact sets.
If $\mathcal{N}^s$ is converging on $\cc^\infty(K_s)$ for $s=1,2$ then $\mathcal{N}^1\ntimes \mathcal{N}^2=(\Pi_0, \Pi_1, \dots)$ is converging on $\cc^\infty(K)$. 
\end{corollary}  
\subsection{Examples}\label{secexadiff}
\begin{asparaenum}[(A)]
\item  We turn to the projector considered in Theorem \ref{th:exkerlejalagrleja} taking a Leja sequence for $a^s_d$ and a $\Re$-Leja sequence for $b_d$ as is illustrated in Table \ref{tab:pointsLeja}.
\begin{theorem} For all function $f$ $22$-times continuously differentiable on a neighborhood of $K=D(0,1)\times [-1,1]$, the polynomial \((\kg[a_0,\dots,a_d] \ntimes (\lag[b_0,\dots,b_d])(f)\) converges uniformly to $f$ on $K$.  
\end{theorem}
\begin{proof} From \cite[Theorem 3.3 and Theorem 4.3]{manh}, $\kg[a_0,\dots,a_d]$ satisfies the assumption of Corollary \ref{th:convdiffoncconvex2} with $\kappa_1=D(0,1)$, $\overline{m}=1$ and $m=4$. On the other hand, according to \cite[Theorem 3.1]{calmanh2}, the Lebesgue constant for $\lag[b_0,\dots,b_d]$ grows at most like $d^3\ln d$ so that convergence holds for four times continuously differentiable function and we may also take $m=4$ (whereas any $\overline{m}\geq 0$ works. Thus Corollary \ref{th:convdiffoncconvex2} ensures convergence for $\mathbf{m}=22= 3\times 3/2+4\times 4 + 3/2$.  
\end{proof}
\item When we work with $\kg[a_0,\dots,a_d] \ntimes \kg[a_0,\dots,a_d]$ for functions on $D(0,1)^2$ in $\RR^4$, the level of required differentiability for applying our result will be $24$. 
\item If we substitute the Kergin interpolants by another related projector known as Hakopian interpolants, in the above, see \cite[Theorem 4.5]{manh}, the level will be $28$. 
\item It is obvious that, in these examples, the levels of differentiability we obtain are rough. In these cases, the available algebraic formulas are very rich and one should be able to derive \textsl{ad hoc} error formulas leading to better results. We hope that our result will be an incentive to further research in this direction. 
\end{asparaenum}

\bibliographystyle{acm}
\bibliography{bibnewtonproduct}
\end{document}